\documentclass[11pt]{amsart}
\usepackage{amssymb, amsmath, latexsym, amsthm, euler, amsxtra, eufrak}
\usepackage[all]{xy}
\usepackage{mathrsfs}
\usepackage{graphicx}

\usepackage{pstcol, pstricks}
\usepackage{stmaryrd}
\usepackage{hyperref}

\newcommand{\comments}[1]{}

\newcommand\linebox{\pspolygon[fillstyle=vlines*,hatchsep=2pt,fillcolor=lightgray,linestyle=dotted](0,0)(.5,0)(.5,.5)(0,.5)(0,0)}
\newcommand\lineboxs{\pspolygon[fillstyle=vlines,hatchsep=2pt,linestyle=none](0,0)(.5,0)(.5,.5)(0,.5)(0,0)}
\newcommand\onebox[1]{\pspolygon[fillstyle=solid,fillcolor=lightgray,linestyle=solid,linewidth=.3mm](0,0)(.5,0)(.5,.5)(0,.5)(0,0)\rput(.25,.25){$#1$}}
\newcommand\circled[1]{\pscirclebox[framesep=1pt,linewidth=.3pt]{\footnotesize#1}}
\newcommand\labeledlinebox[1]{\pspolygon[fillstyle=vlines*,hatchsep=2pt,fillcolor=lightgray,linestyle=dotted](0,0)(.5,0)(.5,.5)(0,.5)(0,0)
\rput(.25,.25){\pscirclebox[fillstyle=solid,fillcolor=lightgray,framesep=1pt,linewidth=.3pt]{\footnotesize#1}}}
\newcommand\txcircled[1]{\begin{pspicture}(0,.5)\pscirclebox[framesep=1pt,linewidth=.3pt]{\footnotesize#1}\end{pspicture}}

\newtheorem{theorem}{Theorem}[section]

\newtheorem{proposition}[theorem]{Proposition}
\newtheorem{corrolary}[theorem]{Corollary}

\theoremstyle{definition}
\newtheorem{example}[theorem]{Example}

\newtheorem{convention}[theorem]{Convention}

\newcommand\C{\mathcal C}
\newcommand\E{\mathcal E}
\newcommand\F{\mathcal F}
\newcommand\OO{\mathcal O}

\newcommand\CC{{\mathbb C}}
\newcommand\PP{{\mathbb P}}
\newcommand\ZZ{{\mathbb Z}}

\newcommand\D{{\scriptscriptstyle \operatorname{D}}}
\newcommand\diff{\operatorname{d}}
\newcommand\HH{\operatorname{H}}
\newcommand\M{\operatorname{M}}
\newcommand\Sym{\operatorname{S}}
\newcommand\T{\operatorname{T}}
\newcommand\Hom{{\operatorname{Hom}}}
\newcommand\Aut{\operatorname{Aut}}
\newcommand\End{\operatorname{End}}
\newcommand\Ext{{\operatorname{Ext}}}
\newcommand\GL{\operatorname{GL}}
\newcommand\Spec{\operatorname{Spec}}
\newcommand\rank{\operatorname{rank}}
\newcommand\rest{\operatorname{rest}}

\title{Torus action on the moduli spaces of plane sheaves}

\author{Jinwon Choi}
\address{Department of Mathematics, University of Illinois at Urbana-Champaign,
1409 E Green St., Urbana, IL 61801, United States}
\email{choi29@illinois.edu}

\author{Mario Maican}
\address{Institute of Mathematics of the Romanian Academy,
Calea Grivitei 21, Bucharest 010702, Romania}
\email{mario.maican@imar.ro}


\begin{document}
\begin{abstract}
We describe the torus fixed locus of the moduli space of stable sheaves with Hilbert polynomial $4m+1$ on $\PP^2$.
We determine the torus representation of the tangent spaces at the fixed points,
which leads to the computation of the Betti and Hodge numbers of the moduli space.
\end{abstract}

\maketitle

\section*{Acknowledgements.}

\noindent
The first author would like to thank Sheldon Katz and Kiryong Chung for useful discussions.
The second author was supported by the
Consiliul Na\c{t}ional al Cercet\u{a}rii \c{S}tiin\c{t}ifice, grant PN II--RU 169/2010 PD--219.

\section{Introduction}

Let $\M = \M_{\PP^2}(4,1)$ denote the moduli space of Gieseker semi-stable sheaves
$\F$ on $\PP^2 = \PP^2(\CC)$ having Hilbert polynomial $P_{\F}(m) = 4m+1$.
According to \cite{lepotier}, $\M$ is an irreducible smooth projective variety of dimension $17$.
Our aim is to classify the torus fixed locus of $\M$ under the natural torus action induced from the torus action on the base space $\PP^2$, which in turn enables us to compute the additive structure of its homology groups with coefficients in $\ZZ$. For this we will use the theory of Bia{\l}ynicki-Birula \cite{bb, bb2, carrell},
which we review in Section \ref{bbtheory}.

More precisely, we will consider the natural action of $T = (\CC^*)^2$
defined as follows: first, $T$ acts on $\PP^2$ by
$(t_1, t_2) \cdot (x_0, x_1, x_2) = (x_0, t_1^{-1} x_1, t_2^{-1} x_2)$;
denote by $\mu_{t} \colon \PP^2 \to \PP^2$ the map of multiplication by $t \in T$.
Now put $t [\F] = [\mu_{t^{-1}}^* \F]$, where $[\F]$ denotes the stable-equivalence class
of the sheaf $\F$.  For this action we will prove the following theorem:

\begin{theorem}
\label{main_theorem}
The fixed point locus of $\M_{\PP^2}(4,1)$ consists of
$180$ isolated points and $6$ one-dimensional components isomorphic to $\PP^1$.
Furthermore, the integral homology of $\M_{\PP^2}(4,1)$ has no torsion and its Poincar{\'e} polynomial is
\begin{multline*}
P_{\M_{\PP^2}(4,1)}(x) = 1+2x^2+6x^4+10x^6+14x^8+15x^{10}+16x^{12}+16x^{14}+16x^{16} \\
+16x^{18}+16x^{20}+16x^{22}+15x^{24}+14x^{26}+10x^{28}+6x^{30}+2x^{32}+x^{34}.
\end{multline*}
\end{theorem}
The geometry of the moduli space $\M$ has been studied by many authors \cite{choi-chung, drezet-maican, sahin, yuan}. In \cite{katz_gv}, it was conjectured that genus zero Gopakumar-Vafa (or BPS) invariant defined in M-theory is equal to the Euler characteristic of the moduli space $\M_{\PP^2}(r,1)$ up to sign. When $r\le 3$, the moduli spaces are well known by the work of Le Potier \cite{lepotier}. When $r=4$, the conjecture was first checked in \cite{sahin} where the author uses a stratification of the moduli space with respect to the global section spaces. The Poincar{\'e} polynomials of the moduli spaces when $r=4$ and $5$ have been computed in \cite{choi-chung} by a wall-crossing technique in the moduli spaces of $\alpha$-stable pairs, and also in \cite{yuan} by the classification of the semi-stable sheaves carried out in \cite{drezet-maican} and \cite{illinois}. Recently, a B-model calculation in physics computes the Poincar{\'e} polynomial up to $r=7$ \cite[Table 2]{hkk} in terms of the refined BPS indices. Mathematically, this calculation can be considered as a conjecture. More mathematical treatment for the refined BPS index can be found in \cite{ckk}. The Poincar\'{e} polynomial in Theorem \ref{main_theorem} agrees with all these previous works.

We will use two approaches to determine $X^T$.
In Section \ref{fixed_locus_1}, following \cite{kool},
we will regard a $T$-fixed sheaf as a $T$-equivariant sheaf and we will classify all $T$-equivariant sheaves in terms of $T$-representations on each affine open subset.
In Section \ref{fixed_locus_2}, we will use the classification of semi-stable sheaves
on $\PP^2$ with Hilbert polynomial $4m+1$ carried out in \cite{drezet-maican}.

The Poincar\'{e} polynomial can then be computed by analyzing the local structure of the moduli spaces around the fixed locus, that is,
by determining the $T$-action on the tangent spaces at the fixed points. For this, we will use two approaches as well.
In Section \ref{representation_1}, we will use the technique developed in \cite{mnop1}. Using the $T$-representations of sheaves on each affine open subset and the associated \v{C}ech complex, we compute the $T$-representation of the tangent spaces.
In Section \ref{representation_2}, we will exploit the locally closed stratification of $\M_{\PP^2}(4,1)$
found in \cite{drezet-maican}.

The first approach can, in principle, be applied to any non-singular moduli space
$\M_{\PP^2}(r, \chi)$ of semi-stable sheaves on $\PP^2$ with Hilbert polynomial $rm+\chi$,
though, of course, for higher multiplicity the calculations will be considerably more involved.
The second approach can be extended to semi-stable sheaves supported on plane quintics
or sextics, for which the classification has been carried out, cf. \cite{illinois} and \cite{kyoto}.

Theorem \ref{main_theorem} also allows us to compute the Hodge numbers $h^{pq}$ of $\M_{\PP^2}(4,1)$.
According to \cite[Theorem 1]{carrell-lieberman},
for any $\lambda \in N$, $h^{pq} = 0$ if $\vert p-q \vert > \dim(X^{\lambda})$.
Choosing a generic $\lambda$ we see that $h^{pq} = 0$ if $\vert p-q \vert > \dim(X^T)=1$.
We obtain the following:

\begin{proposition}
\label{hodge_numbers}
The Hodge numbers $h^{pq}$, $0 \le p, q \le 17$, of $\M_{\PP^2}(4,1)$ satisfy the relations
\[
h^{pq}=0 \quad \text{if} \quad p \neq q \qquad \text{and} \qquad h^{pp} = b_{2p},
\]
where $b_{2p}$ is the Betti number obtained in Theorem \ref{main_theorem}. 
\end{proposition}

\section{A review of the Bia{\l}ynicki-Birula theory}
\label{bbtheory}
As a preliminary, we briefly review the theory of Bia{\l}ynicki-Birula which will be used throughout the paper.
Let $X$ be a smooth projective variety with an action of a torus $T$,
such that the fixed point locus is non-empty.
As usual, we denote by $M$ the group of characters of $T$,
and by $N$ the group of one-parameter subgroups of $T$.
Consider $\lambda \in N$ and the associated $\CC^*$-action on $X$ defined by
$(t, x) \mapsto \lambda(t)\cdot x$.
Let $X_1^\lambda, \ldots, X_r^\lambda$ denote the irreducible components of the
$\CC^*$-fixed point locus $X^\lambda$. They are smooth subvarieties.
We have a \emph{plus decomposition}
\[
X = X_1^{\lambda +} \cup \ldots \cup X_r^{\lambda +}
\]
of $X$ into \emph{plus cells}
\[
X_i^{\lambda +} =\{x\in X \mid \lim_{t\to 0} \lambda(t)\cdot x \in X_i^\lambda \}.
\]
Analogously, we have a \emph{minus decomposition} of $X$ into \emph{minus cells}
\[
X_i^{\lambda -} =\{x\in X \mid \lim_{t \to \infty} \lambda(t)\cdot x \in X_i^\lambda \}.
\]
The plus and minus cells are bundles on $X^{\lambda}$.
More precisely, for each $i$ the restricted tangent bundle ${\T_X}_{\mid X_i^{\lambda}}$
can be decomposed as a direct sum of sub-bundles,
\[
{\T_X}_{\mid X_i^{\lambda}} = \T_i^{+} \oplus \T_i^0 \oplus \T_i^{-},
\]
on which $\CC^*$ acts with positive, zero, respectively negative weights.
Then $X_i^{\lambda +}$ is isomorphic to the bundle space of $\T_i^{+}$ and
$X_i^{\lambda -}$ is isomorphic to the bundle space of $\T_i^{-}$.
Denote $p(i) = \rank(\T_i^{+})$, $n(i) = \rank(\T_i^{-})$.
From the plus and minus decompositions we obtain the Homology Basis Formula \cite[Theorem 4.4]{carrell}:

\begin{theorem}
\label{homology_basis}
For any integer $m$ with $0 \le m \le 2 \dim(X)$, we have a decomposition
\[
\HH_m (X, \ZZ) \simeq \bigoplus_{1 \le i \le r} \HH_{m - 2p(i)}(X_i^{\lambda}, \ZZ)
\simeq \bigoplus_{1 \le i \le r} \HH_{m - 2n(i)}(X_i^{\lambda}, \ZZ).
\]
Thus, the Poincar\'e polynomial $P_X$ of $X$ satisfies the relation
\[
P_X (x) = \sum_{i=1}^r P_{X_i^{\lambda}}(x) x^{2 p(i)}
= \sum_{i=1}^r P_{X_i^{\lambda}}(x) x^{2 n(i)}.
\]
\end{theorem}

According to \cite[Lemma 4.1]{carrell}, for generic $\lambda \in N$ we have $X^T = X^{\lambda}$.
In fact, this is true for all $\lambda$ satisfying the condition:
$\langle \lambda, \chi \rangle \neq 0$ for all $\chi \in M$ occurring in the weight decomposition
of any tangent space $\T_{X,x}$ at a fixed point $x \in X^T$.
Thus, in order to determine the additive structure of $\HH_*(X, \ZZ)$, we need to determine $X^T$,
to find the weight decomposition of the tangent spaces at the fixed points and to compute
$p(i)$ (or $n(i)$) for a one-parameter subgroup $\lambda$ satisfying the above condition.

\section{Torus fixed locus I}
\label{fixed_locus_1}

This and the next section are based on \cite[Chapter 3]{choi_thesis}. By \cite[Proposition 4.4]{kool}, a stable sheaf on $\PP^2$ is $T$-fixed if and only if it is $T$-equivariant. In this section, we use a classification of $T$-equivariant sheaves on a toric variety to study the torus fixed locus of the moduli space of sheaves on the projective plane.

Let the torus $T = (\CC^*)^2$ act on $\PP^2$ by
\[
(t_1,t_2)\cdot(x_0,x_1,x_2)=(x_0,t_1^{-1}x_1,t_2^{-1}x_2).
\]
We consider the standard open affine cover $\{ U_{\alpha} \}$, $\alpha=0,1,2$, of $\PP^2$ invariant under the torus action:
\[
U_\alpha=\{(x_0,x_1,x_2)\in \PP^2 \mid x_\alpha\ne 0\}.
\]
Then we may write
\begin{align*}
U_0 &= \Spec  \CC[x,y] \\
U_1 &= \Spec  \CC[x^{-1},x^{-1}y]\\
U_2 &= \Spec \CC[y^{-1},xy^{-1}],
\end{align*}
where the induced $T$-action on each coordinate ring is given by
$
(t_1,t_2)\cdot(x,y)=(t_1 x, t_2 y).
$

Let $R_\alpha$ denote the coordinate ring $\Gamma(U_\alpha)$, and $v_\alpha$,
$w_\alpha$ denote the $T$-characters for the generators of $R_\alpha$, in other words,
\begin{align*}
(v_0,w_0)&=(t_1,t_2), \\
(v_1,w_1)&=(t_1^{-1},t_1^{-1}t_2), \\
(v_2,w_2)&=(t_2^{-1},t_1t_2^{-1}).
\end{align*}
We let $M^\alpha$ be a copy of the character group $M=\Hom(T,\CC^*)\simeq \ZZ^2$
whose elements are expressed with the basis $\{v_\alpha,w_\alpha\}$.
For $m,m'\in M^\alpha$, we say $m'\geq m$
if every component of $m'-m$ is non-negative, in other words, if $m'-m$ is a character of $R_\alpha$.

We draw $M^\alpha$ so that we can encode the gluing conditions easily as follows.
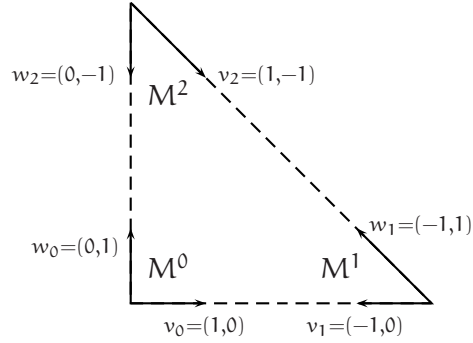
\begin{center}
\begin{pspicture}(0,-.5)(4,4.2)%\psgrid[gridwidth=0.1pt,subgridwidth=0.1pt,gridcolor=red,subgridcolor=green]
%\psset{unit=0.5}
\pspolygon[linestyle=dashed](0,0)(4,0)(0,4)(0,0)
\psline{->}(0,0)(1,0)\psline{->}(0,0)(0,1)
\psline{->}(4,0)(3,0)\psline{->}(4,0)(3,1)
\psline{->}(0,4)(1,3)\psline{->}(0,4)(0,3)
\uput[d](1,0){$^{v_0=(1,0)}$}\uput[dl](0,1){$^{w_0=(0,1)}$}
\uput[d](3,0){$^{v_1=(-1,0)}$}\uput[r](3,1){$^{w_1=(-1,1)}$}
\uput[r](1,3){$^{v_2=(1,-1)}$}\uput[l](0,3){$^{w_2=(0,-1)}$}
\rput(.5,.5){$M^0$}\rput(2.8,.5){$M^1$}\rput(.5,2.8){$M^2$}
\end{pspicture}\end{center}

Let $\F$ be a pure $T$-equivariant sheaf on $\PP^2$. We have a decomposition into
weight spaces
\[
\Gamma(U_{\alpha},\F)=\bigoplus_{m\in M^\alpha}\Gamma(U_{\alpha},\F)_{m}.
\]
Denote the weight space $\Gamma(U_{\alpha},\F)_{m}$ by $F^\alpha(m)$.
Since $\F$ is an $\OO_{\PP^2}$-module, each
$\Gamma(U_{\alpha},\F)$ is an $M^\alpha$-graded $R_{\alpha}$-module.
We can reformulate the $R_{\alpha}$-module structure by the following data:
linear maps  $\chi^\alpha_{m,m'} \colon F^\alpha(m)\to F^\alpha(m')$ for
all $m, m'\in M^i$ with $m'\geq m$ such that
\begin{equation}
\chi^i_{m,m}=1 \quad \text{and} \quad \chi^i_{m,m''}=\chi^i_{m',m''}\circ
\chi^i_{m,m'}.
\end{equation}

The pure one-dimensional $T$-equivariant sheaf $\F$ is supported on the union of three torus fixed lines in $\PP^2$.
They are in a one-to-one correspondence with the collection of weight spaces and linear maps.

In the following theorem, for each $\alpha\in\{1,2,3\}$, let $\beta_1,\beta_2\in\{1,2,3\}\setminus\{\alpha\}$
be such that $v_\alpha^{-1}$ is among the $T$-characters of $ R_{\beta_1}$,
and $w_\alpha^{-1}$ is among the $T$-characters of $ R_{\beta_2}$.
For example, if $\alpha=0$, then $\beta_1=1$ and $\beta_2=2$.

\begin{theorem}[{\cite[Chapter 2]{kool}}]
\label{thm:classifyp2}
The category of pure one-dimensional equivariant sheaves $\F$ on $\PP^2$
is equivalent to the category $\C$ that can be described as follows.
An object of $\C$ is a collection of weight spaces and linear maps between weight spaces
\[
\{F^\alpha(m), \chi^\alpha_{m,m'} \mid m\in M^\alpha, \alpha=1,2,3\},
\]
as described above, which satisfies the following conditions:
\begin{enumerate}
\item   For $i=1,2$, there are integers $A_{\beta_i}$, $A_{\alpha\beta_i}$, and $B_{\alpha\beta_i}$
such that $F^\alpha(m)=0$ unless
    \[ m \in [A_{\beta_1},\infty)\times [A_{\alpha\beta_1},B_{\alpha\beta_1}]\cup
    [A_{\alpha\beta_2},B_{\alpha\beta_2}]\times[A_{\beta_2},\infty). \]
We assume that $A_{\beta_1}\le A_{\alpha\beta_2}$, $A_{\beta_2}\le A_{\alpha\beta_1}$,
$A_{\alpha\beta_i}$ are maximally chosen, and $B_{\alpha\beta_i}$'s are minimally chosen.
\comments{
The set $[A_{\beta_1},\infty)\times [A_{\alpha\beta_1},B_{\alpha\beta_1}]$ or
$[A_{\alpha\beta_2},B_{\alpha\beta_2}]\times[A_{\beta_2},\infty)$ can be an empty set, in which case,
we assume  that $(A_{\beta_1}, A_{\alpha\beta_1},B_{\alpha\beta_1})=(\infty,\infty,\infty)$ or
$(A_{\beta_2}, A_{\alpha\beta_2},B_{\alpha\beta_2})=(\infty,\infty,\infty)$ respectively.}
It is possible that \\
$(A_{\beta_1}, A_{\alpha\beta_1},B_{\alpha\beta_1})=(\infty,\infty,\infty)$ or
$(A_{\beta_2}, A_{\alpha\beta_2},B_{\alpha\beta_2})=(\infty,\infty,\infty)$, which means that the set
$[A_{\beta_1},\infty)\times [A_{\alpha\beta_1},B_{\alpha\beta_1}]$, respectively \\
$[A_{\alpha\beta_2},B_{\alpha\beta_2}]\times[A_{\beta_2},\infty)$ is empty.

\item Assume $(m_1,m_2)\notin [A_{\alpha\beta_2},B_{\alpha\beta_2}]\times [A_{\alpha\beta_1},B_{\alpha\beta_1}]$.
If $(m_1,m_2)$ belongs to $[A_{\beta_1},\infty)\times [A_{\alpha\beta_1},B_{\alpha\beta_1}]$,
then $\chi^\alpha_{(m_1,m_2),(m_1+1,m_2)}$ is injective.
Similarly, if $(m_1,m_2)$ belongs to $[A_{\alpha\beta_2},B_{\alpha\beta_2}]\times[A_{\beta_2},\infty)$,
then $\chi^\alpha_{(m_1,m_2),(m_1,m_2+1)}$ is injective.
Thus, the direct limits
\[
F^{\alpha\beta_1}(m_2):=\displaystyle{\varinjlim_{m_1} F^\alpha(m_1,m_2)} \quad
\text{and} \quad  F^{\alpha\beta_2}(m_1):=\displaystyle{\varinjlim_{m_2} F^\alpha(m_1,m_2)}
\]
are well-defined.
They are required to be finite dimensional vector spaces.
\item For $ m \in [A_{\alpha\beta_2},B_{\alpha\beta_2}]\times [A_{\alpha\beta_1},B_{\alpha\beta_1}]$, the map
\[
\chi^\alpha_{m,(m_1,B_{\alpha\beta_1}+1)}\oplus\chi^\alpha_{m,(B_{\alpha\beta_2}+1,m_2)} \colon
F^\alpha(m)\to F^\alpha(m_1,B_{\alpha\beta_1}+1) \oplus F^\alpha(B_{\alpha\beta_2}+1,m_2)
\]
is injective.
\item For $m_2\in [A_{\alpha\beta_1},B_{\alpha\beta_1}]$,
  \[F^{\alpha\beta_1}(m_2)\simeq F^{\beta_1\alpha}(m_2)\]
  and under this identification, \[\chi^\alpha_{(\infty,m_2),(\infty,m_2+1)}=\chi^{\beta_1}_{(\infty,m_2),(\infty,m_2+1)},\] where
  $\chi^\alpha_{(\infty,m_2),(\infty,m_2+1)}=\displaystyle{\varinjlim_{m_1}
  \chi^\alpha_{(m_1,m_2),(m_1,m_2+1)}}$.

  An analogous statement holds for $\beta_2$.
\end{enumerate}
A morphism
\[
\phi\colon\{F^\alpha(m), \chi^\alpha_{m,m'}\}\to \{G^\alpha(m),
\lambda^\alpha_{m,m'}\}
\] in  $\C$ is a collection of linear maps
$\phi^\alpha(m)\colon F^\alpha(m)\to G^\alpha(m)$ which commute with
$\chi^\alpha$ and $\lambda^\alpha$ such that
\[
\phi^{\alpha\beta_1}(m_2)=\phi^{\beta_1\alpha}(m_2)
\quad \text{and} \quad \phi^{\alpha\beta_2}(m_1)=\phi^{\beta_2\alpha}(m_1),
\]
with obvious notations.
\end{theorem}

By Theorem \ref{thm:classifyp2}, the dimensions of the weight spaces
of pure equivariant sheaves must satisfy the following conditions:
\begin{enumerate}
\item The dimension of the weight space at position $(m_1,m_2)$
is at least the dimensions of the weight spaces at positions
$(m_1-1,m_2)$ and $(m_1,m_2-1)$.
\item Moreover, if $(m_1,m_2)\in [A_{\alpha\beta_2},B_{\alpha\beta_2}]\times [A_{\alpha\beta_1},B_{\alpha\beta_1}]$,
the dimension of the weight space at the position $(m_1,m_2)$
is at most the sum of the dimensions of the weight spaces at positions
    \[(B_{\alpha\beta_2}+1,m_2) \quad \text{and} \quad (m_1,B_{\alpha\beta_1}+1).\]
\end{enumerate}
We will refer this condition as condition ($\ast$).

We illustrate an equivariant sheaf by putting boxes on $M^\alpha$
labeled by the dimensions of the corresponding weight spaces.
By identifying the asymptotic weight spaces, we consider a sheaf as a collection of strips.

For a convenience in illustration, we consider a box in $M^\alpha$ corresponds to
the lattice point of its corner towards the origin of $M^\alpha$.
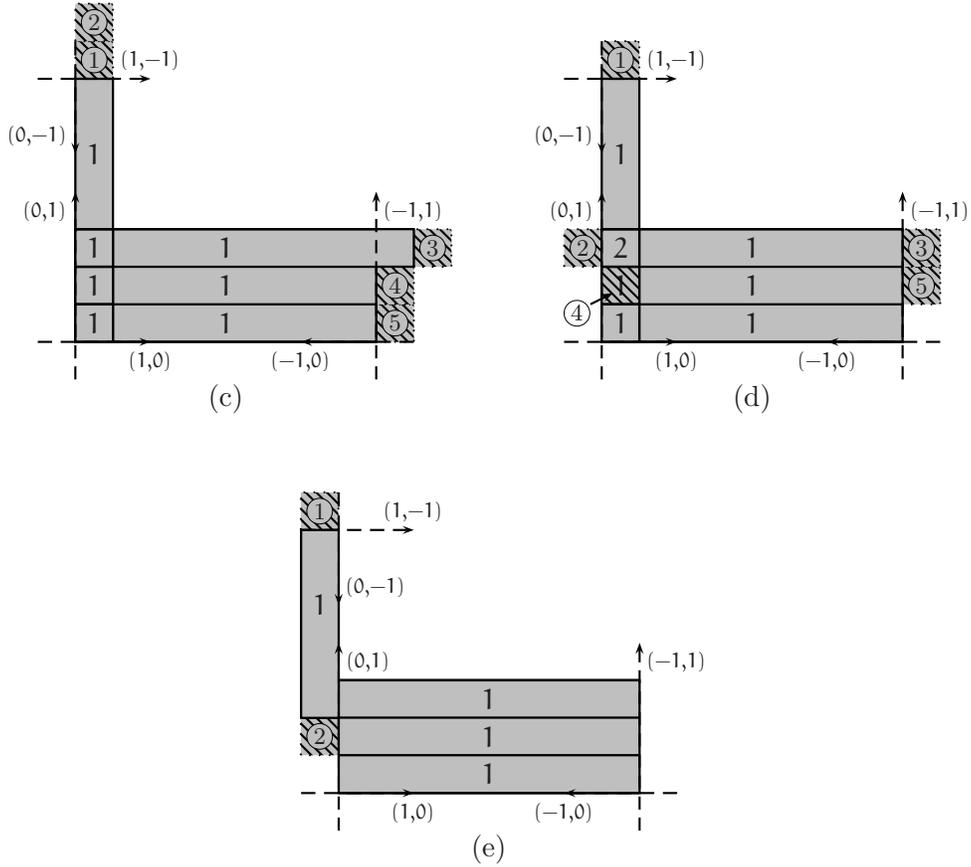
\begin{convention}
A box in $M^0$, $M^1$, $M^2$ respectively corresponds
to the lattice point of its lower left corner, lower right corner, and upper left corner respectively.
\end{convention}

The torus fixed stable sheaves with Hilbert polynomial $dm+\chi$ where $d\le 3$ are described in \cite[Section 2.4]{koolthesis}.
In what follows, we will describe stable $T$-equivariant sheaves with Hilbert polynomial $4m+1$.

\subsection{Case 1. Sheaves supported on a line}

If the sheaf is supported on a line, the problem is the same as
the problem on local $\PP^1$ with $k=1$ studied in \cite{choip1}.
By the discussion in \cite[Section 5.4 and (16)]{choip1},
we have $7$ equivariant sheaves supported on a fixed line.
Since there are three $T$-invariant lines, the contribution from sheaves of this type is $21$.
Examples of $T$-equivariant sheaves supported on a line are depicted as follows.
\medskip

%\begin{figure}
\begin{center}
\begin{pspicture}(0,-1)(12,3)%\psgrid[gridwidth=0.1pt,subgridwidth=0.1pt,gridcolor=red,subgridcolor=green]
\rput(0.5,0){
\pspolygon[fillstyle=solid,fillcolor=lightgray,linestyle=none](0,0.5)(4,.5)(4,1.5)(4.5,1.5)(4.5,2.5)(0,2.5)(0,0)
\rput(4,1){\linebox}
\rput(4.5,2){\linebox}
\rput(-.5,2){\linebox}
\psline[linewidth=.3mm](0,.5)(4,.5)
\rput(0,.5){\psline[linewidth=.3mm](0,.5)(4,.5)}
\rput(0,1){\psline[linewidth=.3mm](0,.5)(4.5,.5)}
\rput(0,1.5){\psline[linewidth=.3mm](0,.5)(4.5,.5)}
\rput(0,2){\psline[linewidth=.3mm](0,.5)(4.5,.5)}
\psline[linewidth=.3mm](4,.5)(4,1.5)
\psline[linewidth=.3mm](4.5,1.5)(4.5,2.5)
\psline[linewidth=.3mm](0,.5)(0,2.5)
\rput(2,.75){$1$}
\rput(2,1.25){$1$}
\rput(2,1.75){$1$}
\rput(2,2.25){$1$}
}
%coordinate axis
\psline[linestyle=dashed]{->}(0,.5)(1.5,.5)\psline[linestyle=dashed]{->}(5,.5)(3.5,.5)
\psline[linestyle=dashed]{->}(0.5,0)(.5,3)
\psline[linestyle=dashed]{->}(4.5,0)(4.5,3)
%\rput(2.3,2.7){$U_{\sigma_1}$}\rput(3.8,2.7){$U_{\sigma_2}$}
\rput(1.5,0.2){$^{(1,0)}$}\rput(0.1,2.7){$^{(0,1)}$}
\rput(3.5,0.2){$^{(-1,0)}$}\rput(5,2.7){$^{(-1,1)}$}
%
%%%%%%%%%% second diagram
\rput(7,0){
\pspolygon[fillstyle=solid,fillcolor=lightgray,linestyle=none](0,0.5)(4,.5)(4,1.5)(4.5,1.5)(4.5,2)(-.5,2)(-.5,1.5)(0,1.5)(0,0)
\psline[linewidth=.3mm](0,.5)(4,.5)
\rput(0,.5){\psline[linewidth=.3mm](0,.5)(4,.5)}
\rput(0,1){\psline[linewidth=.3mm](0,.5)(4.5,.5)}
\rput(0,1.5){\psline[linewidth=.3mm](0,.5)(4.5,.5)}
\psline[linewidth=.3mm](4,.5)(4,1.5)
\psline[linewidth=.3mm](0,.5)(0,1.5)
\rput(4,1.5){\onebox{1}}\rput(-.5,1.5){\onebox{1}}
\rput(2,.75){$1$}
\rput(2,1.25){$1$}
\rput(2,1.75){$2$}
}
%coordinate axis
\rput(6.5,0){
\psline[linestyle=dashed]{->}(0,.5)(1.5,.5)\psline[linestyle=dashed]{->}(5,.5)(3.5,.5)
\psline[linestyle=dashed]{->}(0.5,0)(.5,3)
\psline[linestyle=dashed]{->}(4.5,0)(4.5,3)
%\rput(2.3,2.7){$U_{\sigma_1}$}\rput(3.8,2.7){$U_{\sigma_2}$}
\rput(1.5,0.2){$^{(1,0)}$}\rput(0.1,2.7){$^{(0,1)}$}
\rput(3.5,0.2){$^{(-1,0)}$}\rput(5,2.7){$^{(-1,1)}$}
}%% second diagram end
\rput(2.5,-.7){(a) type (1,1,1,1)}
\rput(9,-.7){(b) type (1,1,2)}
\end{pspicture}
\end{center}
%\caption{Degree 4 sheaves supported on one line}\label{fig:supported1}
%\end{figure}

In (a), since the sheaf is stable, its quotient has Euler characteristic greater than $1$.
Recall that $C_4$ denotes fourth order thickening of $\PP^1$ in its normal direction.
The Euler characteristic of strips in the sheaf $\OO_{C_4}$ are $1,0,-1,-2$.
Thus, by the stability condition, one more box must be added to the third row of the sheaf $\OO_{C_4}$
and this forces another box on the fourth row by condition ($\ast$).
There are two ways to add these two boxes, either as shown in the picture or on the opposite side.
Since the Euler characteristic of each strip is now $1,0,0,$ and $-1$, we have one more box to be added.
The boxes with diagonal lines show three possible ways to add the last box.
Therefore, we get six $T$-equivariant sheaves of these type.

Consider sheaves as in (b). Since the asymptotic weight space of the third row is two-dimensional,
we need to specify the images of one-dimensional weight spaces at each end.
By stability, their images must be linearly independent,
hence we may assume they are (1,0) and (0,1).
It is easy to see this sheaf is stable and is of Euler characteristic $1$.

The sheaves of other types can be easily seen to be decomposable.
Hence, these are all stable sheaves supported on an irreducible line.

\subsection{Case 2. Sheaves supported on the union of two lines}

First we consider the case where all asymptotic weight spaces are one-dimensional.
The scheme theoretic support of the sheaf is either
the union of a triple line and a simple line or the union of two double lines.
We analyze the former case as shown in the following picture.
\medskip

\begin{center}
\begin{pspicture}(0,-6.5)(12,5)%\psgrid[gridwidth=0.1pt,subgridwidth=0.1pt,gridcolor=red,subgridcolor=green]
\rput(0.5,0){
\rput(0,4){\labeledlinebox{1}}\rput(0,4.5){\labeledlinebox{2}}
\rput(4,.5){\labeledlinebox{5}}\rput(4,1){\labeledlinebox{4}}\rput(4.5,1.5){\labeledlinebox{3}}
\pspolygon[fillstyle=solid,fillcolor=lightgray,linewidth=.3mm](0,0.5)(4,.5)(4,1.5)(4.5,1.5)(4.5,2)(0,2)(0,0.5)
\pspolygon[fillstyle=solid,fillcolor=lightgray,linewidth=.3mm](0,0.5)(.5,.5)(.5,4)(0,4)(0,0.5)
\rput(0,1.5){\onebox{1}}\rput(0,1){\onebox{1}}\rput(0,.5){\onebox{1}}
\psline[linewidth=.3mm](0,.5)(4,.5)
\rput(0,.5){\psline[linewidth=.3mm](0,.5)(4,.5)}
\rput(0,1){\psline[linewidth=.3mm](0,.5)(4,.5)}
\rput(0,1.5){\psline[linewidth=.3mm](0,.5)(4,.5)}
\rput(2,.75){$1$}
\rput(2,1.25){$1$}
\rput(2,1.75){$1$}
\rput(.25,3){$1$}
%coordinate axis
\psline[linestyle=dashed]{->}(-.5,.5)(1,.5)\psline[linestyle=dashed]{->}(4.5,.5)(3,.5)
\psline[linestyle=dashed]{->}(0,0)(0,2.5)
\psline[linestyle=dashed]{->}(4,0)(4,2.5)
%\rput(2.3,2.7){$U_{\sigma_1}$}\rput(3.8,2.7){$U_{\sigma_2}$}
\rput(1,0.2){$^{(1,0)}$}\rput(-0.4,2.2){$^{(0,1)}$}
\rput(3,0.2){$^{(-1,0)}$}\rput(4.5,2.2){$^{(-1,1)}$}
\psline[linestyle=dashed]{->}(-0.5,4)(1,4)\psline[linestyle=dashed]{->}(0,4.5)(0,3)
\rput(1,4.2){$^{(1,-1)}$}\rput(-0.5,3.2){$^{(0,-1)}$}
}
\rput(7.5,0){
\rput(0,4){\labeledlinebox{1}}\rput(4,1){\labeledlinebox{5}}\rput(4,1.5){\labeledlinebox{3}}\rput(-.5,1.5){\labeledlinebox{2}}
\pspolygon[fillstyle=solid,fillcolor=lightgray,linewidth=.3mm](0,0.5)(4,.5)(4,2)(0,2)(0,0.5)
\pspolygon[fillstyle=solid,fillcolor=lightgray,linewidth=.3mm](0,0.5)(.5,.5)(.5,4)(0,4)(0,0.5)
\rput(0,1.5){\onebox{2}}\rput(0,1){\onebox{1}}\rput(0,.5){\onebox{1}}
\rput(0,1){\lineboxs}
\rput(-.33,.88){\circled{4}}\psline{->}(-.15,1)(.15,1.15)
\psline[linewidth=.3mm](0,.5)(4,.5)
\rput(0,.5){\psline[linewidth=.3mm](0,.5)(4,.5)}
\rput(0,1){\psline[linewidth=.3mm](0,.5)(4,.5)}
\rput(0,1.5){\psline[linewidth=.3mm](0,.5)(4,.5)}
\rput(2,.75){$1$}
\rput(2,1.25){$1$}
\rput(2,1.75){$1$}
\rput(.25,3){$1$}
%coordinate axis
\psline[linestyle=dashed]{->}(-.5,.5)(1,.5)\psline[linestyle=dashed]{->}(4.5,.5)(3,.5)
\psline[linestyle=dashed]{->}(0,0)(0,2.5)
\psline[linestyle=dashed]{->}(4,0)(4,2.5)
%\rput(2.3,2.7){$U_{\sigma_1}$}\rput(3.8,2.7){$U_{\sigma_2}$}
\rput(1,0.2){$^{(1,0)}$}\rput(-0.4,2.2){$^{(0,1)}$}
\rput(3,0.2){$^{(-1,0)}$}\rput(4.5,2.2){$^{(-1,1)}$}
\psline[linestyle=dashed]{->}(-0.5,4)(1,4)\psline[linestyle=dashed]{->}(0,4.5)(0,3)
\rput(1,4.2){$^{(1,-1)}$}\rput(-0.5,3.2){$^{(0,-1)}$}
}
\rput(4,-6){
\rput(-.5,1){\labeledlinebox{2}}\rput(-.5,4){\labeledlinebox{1}}%\rput(-.5,4.5){\labeledlinebox{2}}
\pspolygon[fillstyle=solid,fillcolor=lightgray,linewidth=.3mm](0,0.5)(4,.5)(4,2)(0,2)(0,0.5)
\pspolygon[fillstyle=solid,fillcolor=lightgray,linewidth=.3mm](-.5,1.5)(0,1.5)(0,4)(-.5,4)(-.5,1.5)
%\rput(-.5,1.5){\onebox{1}}
\psline[linewidth=.3mm](0,.5)(4,.5)
\rput(0,.5){\psline[linewidth=.3mm](0,.5)(4,.5)}
\rput(0,1){\psline[linewidth=.3mm](0,.5)(4,.5)}
\rput(0,1.5){\psline[linewidth=.3mm](0,.5)(4,.5)}
\rput(2,.75){$1$}
\rput(2,1.25){$1$}
\rput(2,1.75){$1$}
\rput(-.25,3){$1$}
%coordinate axis
\psline[linestyle=dashed]{->}(-.5,.5)(1,.5)\psline[linestyle=dashed]{->}(4.5,.5)(3,.5)
\psline[linestyle=dashed]{->}(0,0)(0,2.5)
\psline[linestyle=dashed]{->}(4,0)(4,2.5)
%\rput(2.3,2.7){$U_{\sigma_1}$}\rput(3.8,2.7){$U_{\sigma_2}$}
\rput(1,0.2){$^{(1,0)}$}\rput(0.4,2.2){$^{(0,1)}$}
\rput(3,0.2){$^{(-1,0)}$}\rput(4.5,2.2){$^{(-1,1)}$}
\psline[linestyle=dashed]{->}(-0.5,4)(1,4)\psline[linestyle=dashed]{->}(0,4.5)(0,3)
\rput(1,4.2){$^{(1,-1)}$}\rput(0.5,3.2){$^{(0,-1)}$}
}
\rput(2.5,-.25){(c)}\rput(9.5,-.25){(d)}\rput(6,-6.25){(e)}
\end{pspicture}
\end{center}

As before, since the sheaf is stable, the third row must contain one more box than the structure sheaf.
There are three possible ways (c), (d), and (e) as shown in the above picture.

The Euler characteristics of (c) and (d) are $-1$, and that of (e) is zero.
Thus, we need to add two more boxes to (c) and (d) and one more box to (e).
The possible places for the additional boxes are shown as labeled boxes with diagonal lines.
Adding a box to an existing one-dimensional weight space means the increase of its dimension to two.
The resulting weight space configurations must satisfy the condition ($\ast$) and the stability condition.
Here are the lists of all possible ways to add two more boxes for each case:
\begin{align*}
  (\textrm{c})&\colon \{\txcircled{1},\txcircled{2}\},\{\txcircled{1},\txcircled{3}\},
  \{\txcircled{1},\txcircled{4}\},\{\txcircled{3},\txcircled{4}\},\{\txcircled{4},\txcircled{5}\},\\
  (\textrm{d})&\colon   \{\txcircled{1},\txcircled{2}\},\{\txcircled{1},\txcircled{3}\},
  \{\txcircled{3},\txcircled{5}\},\{\txcircled{2},\txcircled{4}\},\{\txcircled{3},\txcircled{4}\},\\
  (\textrm{e})&\colon  \{\txcircled{1}\},\{\txcircled{2}\}. \\
\end{align*}

For example, in (d), $\{\txcircled{1},\txcircled{4}\}$ or $\{\txcircled{2},\txcircled{3}\}$
are not allowed by stability, because the subsheaf generated by each of them has Euler characteristic $1$.

Therefore, we have $6\times (5+5+2)=72$ equivariant sheaves of this type.
\medskip

Next, we consider the case where the scheme theoretic support is a union of two double lines.

\medskip

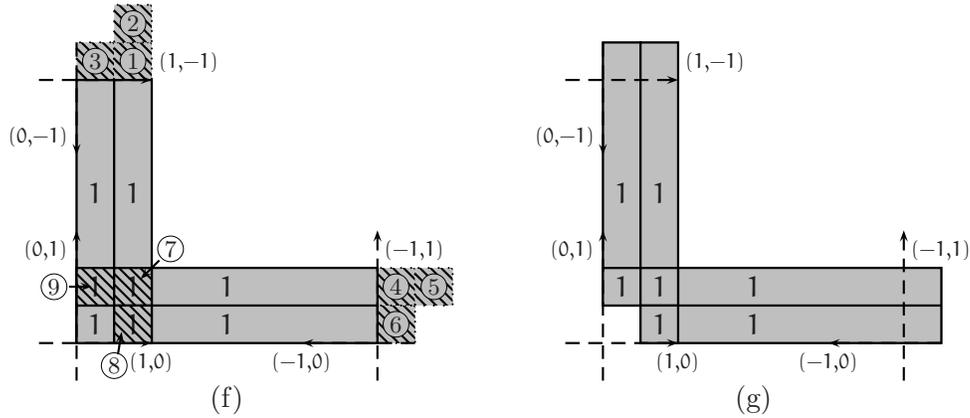
\begin{center}
\begin{pspicture}(0,-.5)(12.5,5)%\psgrid[gridwidth=0.1pt,subgridwidth=0.1pt,gridcolor=red,subgridcolor=green]
\rput(0.5,0){
\rput(4,0){\rput(0,1){\labeledlinebox{4}}\rput(0.5,1){\labeledlinebox{5}}\rput(0,.5){\labeledlinebox{6}}}
\rput(0,4){\rput(0,0){\labeledlinebox{3}}\rput(0.5,0){\labeledlinebox{1}}\rput(0.5,.5){\labeledlinebox{2}}}
\pspolygon[fillstyle=solid,fillcolor=lightgray,linewidth=.3mm](0,0.5)(4,.5)(4,1.5)(0,1.5)(0,0.5)
\pspolygon[fillstyle=solid,fillcolor=lightgray,linewidth=.3mm](0,0.5)(.5,.5)(.5,4)(0,4)(0,0.5)
\rput(0.5,0){\pspolygon[fillstyle=solid,fillcolor=lightgray,linewidth=.3mm](0,0.5)(.5,.5)(.5,4)(0,4)(0,0.5)}
\rput(0,.5){\psline[linewidth=.3mm](0,.5)(4,.5)}
\rput(0,.5){\onebox{1}}\rput(0,1){\onebox{1}}\rput(0.5,1){\onebox{1}}\rput(0.5,.5){\onebox{1}}
\rput(0,1){\lineboxs}\rput(0.5,1){\lineboxs}\rput(0.5,.5){\lineboxs}
\rput(1.25,1.75){\circled{7}}\psline{->}(1.1,1.6)(.85,1.35)
\rput(.5,.2){\circled{8}}\psline{->}(.55,.35)(.6,.65)
\rput(-.3,1.25){\circled{9}}\psline{->}(-.1,1.25)(.2,1.25)
\rput(2,.75){$1$}
\rput(2,1.25){$1$}
\rput(.25,2.5){$1$}
\rput(.75,2.5){$1$}
%coordinate axis
\psline[linestyle=dashed]{->}(-.5,.5)(1,.5)\psline[linestyle=dashed]{->}(4.5,.5)(3,.5)
\psline[linestyle=dashed]{->}(0,0)(0,2)
\psline[linestyle=dashed]{->}(4,0)(4,2)
%\rput(2.3,2.7){$U_{\sigma_1}$}\rput(3.8,2.7){$U_{\sigma_2}$}
\rput(1,0.2){$^{(1,0)}$}\rput(-0.4,1.7){$^{(0,1)}$}
\rput(3,0.2){$^{(-1,0)}$}\rput(4.5,1.7){$^{(-1,1)}$}
\psline[linestyle=dashed]{->}(-0.5,4)(1,4)\psline[linestyle=dashed]{->}(0,4.5)(0,3)
\rput(1.5,4.2){$^{(1,-1)}$}\rput(-0.5,3.2){$^{(0,-1)}$}
}
\rput(7.5,0){
\rput(.5,0){\pspolygon[fillstyle=solid,fillcolor=lightgray,linewidth=.3mm](0,0.5)(4,.5)(4,1.5)(0,1.5)(0,0.5)}
\rput(0,0.5){\pspolygon[fillstyle=solid,fillcolor=lightgray,linewidth=.3mm](0,0.5)(.5,.5)(.5,4)(0,4)(0,0.5)}
\rput(0.5,0.5){\pspolygon[fillstyle=solid,fillcolor=lightgray,linewidth=.3mm](0,0.5)(.5,.5)(.5,4)(0,4)(0,0.5)}
%\pspolygon[fillstyle=solid,fillcolor=white,linewidth=.3mm](0,0.5)(.5,.5)(.5,1)(0,1)(0,0.5)
\rput(0.5,.5){\psline[linewidth=.3mm](0,.5)(4,.5)}
\rput(0,1){\onebox{1}}\rput(0.5,1){\onebox{1}}\rput(0.5,.5){\onebox{1}}
\rput(2,.75){$1$}
\rput(2,1.25){$1$}
\rput(.25,2.5){$1$}
\rput(.75,2.5){$1$}
%coordinate axis
\psline[linestyle=dashed]{->}(-.5,.5)(1,.5)\psline[linestyle=dashed]{->}(4.5,.5)(3,.5)
\psline[linestyle=dashed]{->}(0,0)(0,2)
\psline[linestyle=dashed]{->}(4,0)(4,2)
%\rput(2.3,2.7){$U_{\sigma_1}$}\rput(3.8,2.7){$U_{\sigma_2}$}
\rput(1,0.2){$^{(1,0)}$}\rput(-0.4,1.7){$^{(0,1)}$}
\rput(3,0.2){$^{(-1,0)}$}\rput(4.5,1.7){$^{(-1,1)}$}
\psline[linestyle=dashed]{->}(-0.5,4)(1,4)\psline[linestyle=dashed]{->}(0,4.5)(0,3)
\rput(1.5,4.2){$^{(1,-1)}$}\rput(-0.5,3.2){$^{(0,-1)}$}
}
\rput(2.5,-.25){(f)}\rput(9.5,-.25){(g)}
\end{pspicture}
\end{center}

If we remove one box from the structure sheaf as in (g), the stability forces boxes on the other sides.
By condition ($\ast$), the sheaf in (g) is the only possible one.
In (f), we need to add three more boxes to the structure sheaf.
As before, possible places are shown by labeled boxes with diagonal lines.
There are 12 possible ways to add these three boxes without violating the stability as follows:
\begin{align*}
  (\textrm{f})\colon &\{\txcircled{1},\txcircled{2},\txcircled{3}\},\{\txcircled{1},\txcircled{2},\txcircled{4}\},
  \{\txcircled{1},\txcircled{3},\txcircled{4}\},  \{\txcircled{1},\txcircled{4},\txcircled{5}\},
  \{\txcircled{1},\txcircled{4},\txcircled{6}\},   \{\txcircled{4},\txcircled{5},\txcircled{6}\},\\
  &  \{\txcircled{1},\txcircled{3},\txcircled{7}\},  \{\txcircled{1},\txcircled{4},\txcircled{7}\},
  \{\txcircled{4},\txcircled{6},\txcircled{7}\},  \{\txcircled{1},\txcircled{7},\txcircled{9}\},
  \{\txcircled{4},\txcircled{7},\txcircled{8}\},  \{\txcircled{7},\txcircled{8},\txcircled{9}\},
\end{align*}

Similarly as in the previous case, examples of adding three boxes on the same line such as
$\{\txcircled{1},\txcircled{2},\txcircled{7}\}$ and $\{\txcircled{1},\txcircled{7},\txcircled{8}\}$
are not allowed by stability. So, there are $3\times (12+1)=39$ equivariant sheaves of this type.
\medskip

There are two more equivariant sheaves, that have two-dimensional asymptotic
weight spaces illustrated below, which leads to the contribution
$6\times 2=12$. We can check that no other weight space configurations are possible.
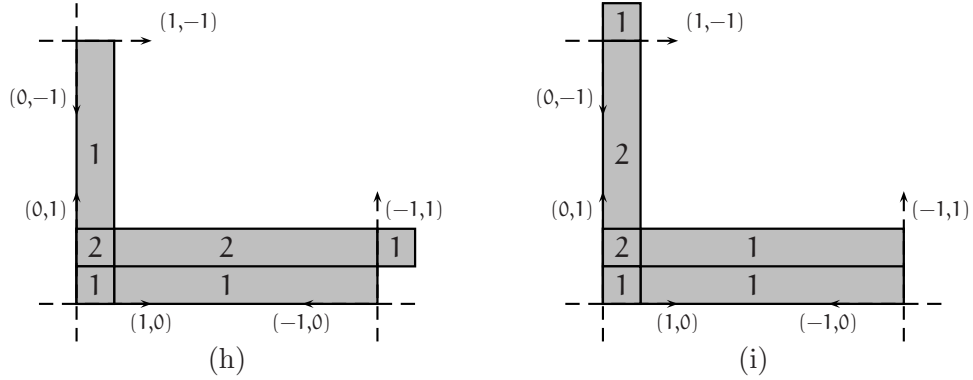
\begin{center}
\begin{pspicture}(0,-.5)(12.5,5)%\psgrid[gridwidth=0.1pt,subgridwidth=0.1pt,gridcolor=red,subgridcolor=green]
\rput(0.5,0){
\pspolygon[fillstyle=solid,fillcolor=lightgray,linewidth=.3mm](0,0.5)(4,.5)(4,1.5)(0,1.5)(0,0.5)
\pspolygon[fillstyle=solid,fillcolor=lightgray,linewidth=.3mm](0,0.5)(.5,.5)(.5,4)(0,4)(0,0.5)
\rput(0,.5){\psline[linewidth=.3mm](0,.5)(4,.5)}
\rput(4,1){\onebox{1}}\rput(0,1){\onebox{2}}\rput(0,.5){\onebox{1}}
\rput(2,.75){$1$}
\rput(2,1.25){$2$}
\rput(.25,2.5){$1$}
%coordinate axis
\psline[linestyle=dashed]{->}(-.5,.5)(1,.5)\psline[linestyle=dashed]{->}(4.5,.5)(3,.5)
\psline[linestyle=dashed]{->}(0,0)(0,2)
\psline[linestyle=dashed]{->}(4,0)(4,2)
%\rput(2.3,2.7){$U_{\sigma_1}$}\rput(3.8,2.7){$U_{\sigma_2}$}
\rput(1,0.2){$^{(1,0)}$}\rput(-0.4,1.7){$^{(0,1)}$}
\rput(3,0.2){$^{(-1,0)}$}\rput(4.5,1.7){$^{(-1,1)}$}
\psline[linestyle=dashed]{->}(-0.5,4)(1,4)\psline[linestyle=dashed]{->}(0,4.5)(0,3)
\rput(1.5,4.2){$^{(1,-1)}$}\rput(-0.5,3.2){$^{(0,-1)}$}
}
\rput(7.5,0){
\pspolygon[fillstyle=solid,fillcolor=lightgray,linewidth=.3mm](0,0.5)(4,.5)(4,1.5)(0,1.5)(0,0.5)
\pspolygon[fillstyle=solid,fillcolor=lightgray,linewidth=.3mm](0,0.5)(.5,.5)(.5,4)(0,4)(0,0.5)
\rput(0,.5){\psline[linewidth=.3mm](0,.5)(4,.5)}
\rput(0,4){\onebox{1}}\rput(0,1){\onebox{2}}\rput(0,.5){\onebox{1}}
\rput(2,.75){$1$}
\rput(2,1.25){$1$}
\rput(.25,2.5){$2$}
%coordinate axis
\psline[linestyle=dashed]{->}(-.5,.5)(1,.5)\psline[linestyle=dashed]{->}(4.5,.5)(3,.5)
\psline[linestyle=dashed]{->}(0,0)(0,2)
\psline[linestyle=dashed]{->}(4,0)(4,2)
%\rput(2.3,2.7){$U_{\sigma_1}$}\rput(3.8,2.7){$U_{\sigma_2}$}
\rput(1,0.2){$^{(1,0)}$}\rput(-0.4,1.7){$^{(0,1)}$}
\rput(3,0.2){$^{(-1,0)}$}\rput(4.5,1.7){$^{(-1,1)}$}
\psline[linestyle=dashed]{->}(-0.5,4)(1,4)\psline[linestyle=dashed]{->}(0,4.5)(0,3)
\rput(1.5,4.2){$^{(1,-1)}$}\rput(-0.5,3.2){$^{(0,-1)}$}
}
\rput(2.5,-.25){(h)}\rput(9.5,-.25){(i)}
\end{pspicture}
\end{center}

In conclusion, there are $123$ torus fixed sheaves supported on the union of two lines.
\medskip

\subsection{Case 3. Sheaves supported on the union of three lines}

If all the asymptotic weight spaces are one-dimensional,
then the schematic support of the sheaf is a union of two simple lines and a double line, which we denote by $C$.
The sheaf in this case is obtained by either adding three boxes to $\OO_C$
or removing one box from $\OO_C(1)$. We start with the former case.

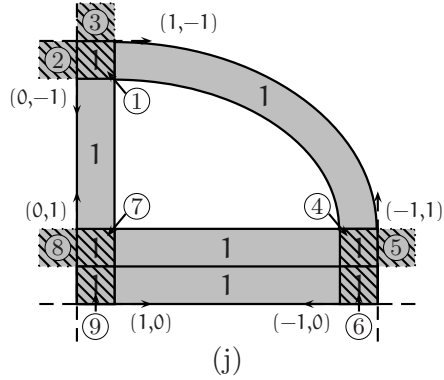
\begin{center}
\begin{pspicture}(0,-.5)(5.5,5)%\psgrid[gridwidth=0.1pt,subgridwidth=0.1pt,gridcolor=red,subgridcolor=green]
\rput(0.5,0){
\psellipse[fillstyle=solid,fillcolor=lightgray,linewidth=.3mm](.5,1.5)(3.5,2.5)
\psellipse[fillstyle=solid,fillcolor=white,linewidth=.3mm](.5,1.5)(3,2)
\pspolygon[fillstyle=solid,fillcolor=white,linestyle=none](.5,1.5)(.5,4)(-5,4)(-5,-5)(5,-5)(5,1.5)(.5,1.5)
\pspolygon[fillstyle=solid,fillcolor=lightgray,linewidth=.3mm](0,0.5)(4,.5)(4,1.5)(0,1.5)(0,0.5)
\pspolygon[fillstyle=solid,fillcolor=lightgray,linewidth=.3mm](0,0.5)(.5,.5)(.5,4)(0,4)(0,0.5)
\rput(0,.5){\psline[linewidth=.3mm](0,.5)(4,.5)}
\rput(-.5,1){\labeledlinebox{8}}\rput(4,1){\labeledlinebox{5}}\rput(-.5,3.5){\labeledlinebox{2}}\rput(0,4){\labeledlinebox{3}}
\rput(0,1){\onebox{1}}\rput(0,.5){\onebox{1}}\rput(3.5,1){\onebox{1}}\rput(3.5,.5){\onebox{1}}\rput(0,3.5){\onebox{1}}
\rput(0,1){\lineboxs}\rput(0,.5){\lineboxs}\rput(3.5,1){\lineboxs}\rput(3.5,.5){\lineboxs}\rput(0,3.5){\lineboxs}
\rput(.8,3.2){\circled{1}}\psline{->}(.65,3.3)(.4,3.6)
\rput(.8,1.8){\circled{7}}\psline{->}(.65,1.7)(.4,1.4)
\rput(3.2,1.8){\circled{4}}\psline{->}(3.35,1.7)(3.6,1.4)
\rput(3.75,.2){\circled{6}}\psline{->}(3.75,.4)(3.75,.65)
\rput(.25,.2){\circled{9}}\psline{->}(.25,.4)(.25,.65)

\rput(2.5,3.27){$1$}
\rput(2,.75){$1$}
\rput(2,1.25){$1$}
\rput(.25,2.5){$1$}
%coordinate axis
\psline[linestyle=dashed]{->}(-.5,.5)(1,.5)\psline[linestyle=dashed]{->}(4.5,.5)(3,.5)
\psline[linestyle=dashed]{->}(0,0)(0,2)
\psline[linestyle=dashed]{->}(4,0)(4,2)
%\rput(2.3,2.7){$U_{\sigma_1}$}\rput(3.8,2.7){$U_{\sigma_2}$}
\rput(1,0.2){$^{(1,0)}$}\rput(-0.4,1.7){$^{(0,1)}$}
\rput(3,0.2){$^{(-1,0)}$}\rput(4.5,1.7){$^{(-1,1)}$}
\psline[linestyle=dashed]{->}(-0.5,4)(1,4)\psline[linestyle=dashed]{->}(0,4.5)(0,3)
\rput(1.5,4.2){$^{(1,-1)}$}\rput(-0.5,3.2){$^{(0,-1)}$}
}
\rput(2.5,-.25){(j)}
\end{pspicture}
\end{center}

As before, the possible places for three added boxes are shown by boxes with diagonal lines.
By the stability, three boxes added cannot be on the same line,
otherwise the subsheaf generated by these three boxes would have Euler characteristic $1$.
There are $10$ sheaves of this type:
\begin{align*}
  (\textrm{j})\colon &\{\txcircled{1},\txcircled{2},\txcircled{3}\},\{\txcircled{1},\txcircled{2},\txcircled{7}\},
  \{\txcircled{1},\txcircled{3},\txcircled{4}\},  \{\txcircled{1},\txcircled{4},\txcircled{7}\}, \\
  & \{\txcircled{4},\txcircled{5},\txcircled{6}\},\{\txcircled{1},\txcircled{4},\txcircled{5}\},   \{\txcircled{4},\txcircled{6},\txcircled{7}\},\\
  & \{\txcircled{7},\txcircled{8},\txcircled{9}\}, \{\txcircled{1},\txcircled{7},\txcircled{8}\},  \{\txcircled{4},\txcircled{7},\txcircled{9}\},
\end{align*}
Their total contribution is $3\times 10=30$.

Now, we consider the latter case of removing one box from $\OO_C(1)$.

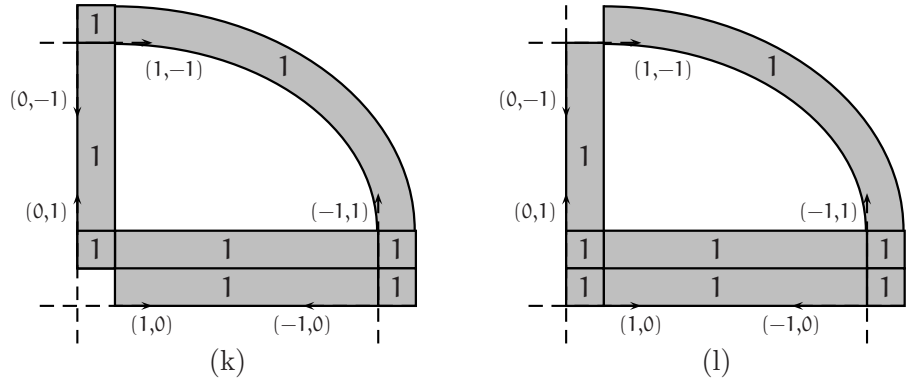
\begin{center}
\begin{pspicture}(0,-0.5)(12,5)%\psgrid[gridwidth=0.1pt,subgridwidth=0.1pt,gridcolor=red,subgridcolor=green]
\rput(7,0){
\psellipse[fillstyle=solid,fillcolor=lightgray,linewidth=.3mm](.5,1.5)(4,3)
\psellipse[fillstyle=solid,fillcolor=white,linewidth=.3mm](.5,1.5)(3.5,2.5)
\pspolygon[fillstyle=solid,fillcolor=white,linestyle=none](.5,1.5)(.5,5)(-5,5)(-5,-5)(5,-5)(5,1.5)(.5,1.5)
\pspolygon[fillstyle=solid,fillcolor=lightgray,linewidth=.3mm](0,0.5)(4,.5)(4,1.5)(0,1.5)(0,0.5)
\pspolygon[fillstyle=solid,fillcolor=lightgray,linewidth=.3mm](0,0.5)(.5,.5)(.5,4)(0,4)(0,0.5)
}
\rput(0.5,0){
\psellipse[fillstyle=solid,fillcolor=lightgray,linewidth=.3mm](.5,1.5)(4,3)
\psellipse[fillstyle=solid,fillcolor=white,linewidth=.3mm](.5,1.5)(3.5,2.5)
\pspolygon[fillstyle=solid,fillcolor=white,linestyle=none](.5,1.5)(.5,5)(-5,5)(-5,-5)(5,-5)(5,1.5)(.5,1.5)
\pspolygon[fillstyle=solid,fillcolor=lightgray,linewidth=.3mm](0.5,0.5)(4,.5)(4,1.5)(0.5,1.5)(0.5,0.5)
\pspolygon[fillstyle=solid,fillcolor=lightgray,linewidth=.3mm](0,1)(.5,1)(.5,4)(0,4)(0,1)
}
\rput(.5,0){
\rput(0,.5){\psline[linewidth=.3mm](0,.5)(4.5,.5)}
\rput(0,1){\onebox{1}}\rput(4,1){\onebox{1}}\rput(4,.5){\onebox{1}}\rput(0,4){\onebox{1}}
\rput(2.75,3.70){$1$}
\rput(2,.75){$1$}
\rput(2,1.25){$1$}
\rput(.25,2.5){$1$}
%coordinate axis
\psline[linestyle=dashed]{->}(-.5,.5)(1,.5)\psline[linestyle=dashed]{->}(4.5,.5)(3,.5)
\psline[linestyle=dashed]{->}(0,0)(0,2)
\psline[linestyle=dashed]{->}(4,0)(4,2)
\rput(1,0.2){$^{(1,0)}$}\rput(-0.4,1.7){$^{(0,1)}$}
\rput(3,0.2){$^{(-1,0)}$}\rput(3.5,1.7){$^{(-1,1)}$}
\psline[linestyle=dashed]{->}(-0.5,4)(1,4)\psline[linestyle=dashed]{->}(0,4.5)(0,3)
\rput(1.3,3.6){$^{(1,-1)}$}\rput(-0.5,3.2){$^{(0,-1)}$}
}
\rput(7,0){
\rput(0,.5){\psline[linewidth=.3mm](0,.5)(4.5,.5)}
\rput(0.5,0){\psline[linewidth=.3mm](0,4)(0,4.5)}
\rput(0,1){\onebox{1}}\rput(0,.5){\onebox{1}}\rput(4,1){\onebox{1}}\rput(4,.5){\onebox{1}}
\rput(2.75,3.70){$1$}
\rput(2,.75){$1$}
\rput(2,1.25){$1$}
\rput(.25,2.5){$1$}
%coordinate axis
\psline[linestyle=dashed]{->}(-.5,.5)(1,.5)\psline[linestyle=dashed]{->}(4.5,.5)(3,.5)
\psline[linestyle=dashed]{->}(0,0)(0,2)
\psline[linestyle=dashed]{->}(4,0)(4,2)
\rput(1,0.2){$^{(1,0)}$}\rput(-0.4,1.7){$^{(0,1)}$}
\rput(3,0.2){$^{(-1,0)}$}\rput(3.5,1.7){$^{(-1,1)}$}
\psline[linestyle=dashed]{->}(-0.5,4)(1,4)\psline[linestyle=dashed]{->}(0,4.5)(0,3)
\rput(1.3,3.6){$^{(1,-1)}$}\rput(-0.5,3.2){$^{(0,-1)}$}
}
\rput(2.5,-.25){(k)}\rput(9,-.25){(l)}
\end{pspicture}
\end{center}

In all sheaves described so far, by changing bases of the weight spaces and using
Theorem \ref{thm:classifyp2}, we can assume the $\chi^\alpha$ maps are either identities,
or projections, or inclusions, depending on the dimensions of the weight spaces involved.
However, we cannot make the same assumption for the sheaf shown at (k) above.
If we start fixing bases from the weight spaces in the lower right corner of the diagram
such that $\chi^\alpha$ are all identities, we have two choices for a basis at the box in the upper left corner which do not need to agree.
So, this weight space configuration will determine an one-dimensional torus fixed locus.
In Example \ref{ex:1dimlocusp2}, we will show that this one-dimensional locus is isomorphic to $\PP^1$.
%We can easily see this by simple dimension count.
Hence, there are $6$ one-dimensional components isomorphic to $\PP^1$ in the torus fixed locus.

It is clear that diagram (l) defines a stable sheaf and there are $3$ of them.

The final example is where we have a two-dimensional asymptotic weight space.
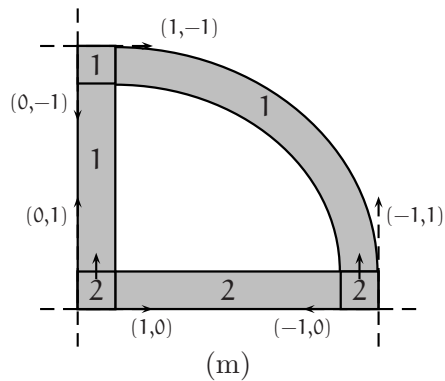
\begin{center}
\begin{pspicture}(0,-0.5)(5,5)%\psgrid[gridwidth=0.1pt,subgridwidth=0.1pt,gridcolor=red,subgridcolor=green]
\rput(0.5,0)
{
\psellipse[fillstyle=solid,fillcolor=lightgray,linewidth=.3mm](.5,1)(3.5,3)
\psellipse[fillstyle=solid,fillcolor=white,linewidth=.3mm](.5,1)(3,2.5)
\pspolygon[fillstyle=solid,fillcolor=white,linestyle=none](.5,1)(.5,4)(-5,4)(-5,-5)(5,-5)(5,1)(.5,1)
\pspolygon[fillstyle=solid,fillcolor=lightgray,linewidth=.3mm](0,0.5)(4,.5)(4,1)(0,1)(0,0.5)
\pspolygon[fillstyle=solid,fillcolor=lightgray,linewidth=.3mm](0,0.5)(.5,.5)(.5,4)(0,4)(0,0.5)
\rput(0,.5){\onebox{2}}\rput(3.5,.5){\onebox{2}}\rput(0,3.5){\onebox{1}}
\rput(2.5,3.2){$1$}
\rput(2,.75){$2$}
\rput(.25,2.5){$1$}
%coordinate axis
\psline[linestyle=dashed]{->}(-.5,.5)(1,.5)\psline[linestyle=dashed]{->}(4.5,.5)(3,.5)
\psline[linestyle=dashed]{->}(0,0)(0,2)
\psline[linestyle=dashed]{->}(4,0)(4,2)
%\rput(2.3,2.7){$U_{\sigma_1}$}\rput(3.8,2.7){$U_{\sigma_2}$}
\rput(1,0.2){$^{(1,0)}$}\rput(-0.4,1.7){$^{(0,1)}$}
\rput(3,0.2){$^{(-1,0)}$}\rput(4.5,1.7){$^{(-1,1)}$}
\psline[linestyle=dashed]{->}(-0.5,4)(1,4)\psline[linestyle=dashed]{->}(0,4.5)(0,3)
\rput(1.5,4.2){$^{(1,-1)}$}\rput(-0.5,3.2){$^{(0,-1)}$}
}
\psline{->}(0.75,.9)(.75,1.25)
\psline{->}(4.25,.9)(4.25,1.25)
\rput(2.5,-.25){(m)}
\end{pspicture}
\end{center}
If the kernels of two $\chi^\alpha$ maps shown by arrows are distinct,
the sheaf is not decomposable, and hence, stable. There are $3$ equivariant sheaves of this kind.

In conclusion, we have the first part of Theorem \ref{main_theorem}.

\begin{theorem}
The $(\CC^*)^2$-fixed point locus of $\M_{\PP^2}(4,1)$ consists of
$180$ isolated points and $6$ one-dimensional components isomorphic to $\PP^1$.
\end{theorem}

\begin{corrolary}
  The topological Euler characteristic of $\M_{\PP^2}(4,1)$ is $192$.
\end{corrolary}

\begin{example}[Positive dimensional fixed locus]
\label{ex:1dimlocusp2}\hfill \\
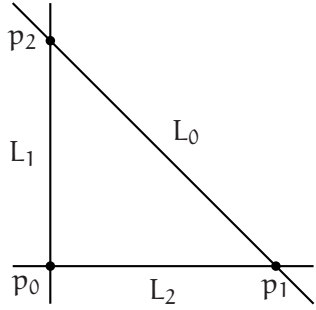
\begin{center}
\begin{pspicture}(4,4)%
\psline(.5,0)(.5,4) \psline(0,.5)(4,.5) \psline(0,4)(4,0)
\psdots(.5,.5)(.5,3.5)(3.5,.5) \uput[l](.5,3.5){$p_2$}
\uput[dl](.5,.5){$p_0$}\uput[d](3.5,.5){$p_1$} \uput[l](.5,2){$L_1$}
\uput[d](2,.5){$L_2$}\uput[ur](2,2){$L_0$}
\end{pspicture}
\end{center}
Let $L_0, L_1, L_2$ be the three torus fixed lines in $\PP^2$. Let $C=2L_2 \cup L_0$.

If we read the diagram of one-dimensional fixed locus above,
we get the stable sheaf defined by the following short exact sequence:
\[
\xymatrix  @C=2em @R=2em @*[c]{0\ar[r] &\F \ar[r] & I_{p_0,C}(1)\oplus \OO_{L_1}\ar[rrr]^-{
\left(
  \begin{smallmatrix}
    -a\cdot \rest_{p_0} &     b\cdot \rest_{p_0} \\
    -c\cdot \rest_{p_2} &     d\cdot \rest_{p_2} \\
  \end{smallmatrix}
\right)}
& &&\CC_{p_0} \oplus \CC_{p_2} \ar[r] &0 }.
\]
Here $a, b, c, d$ are complex numbers and $\rest_{p_0}$, $\rest_{p_2}$ are the restriction
maps to the corresponding points.
By stability, neither $a$ nor $c$ can be zero. Indeed, if $a$ were zero, then $\F$
would be the kernel of the corestriction $I_{p_0,C}(1)\oplus I_{p_0,L_1}\to \CC_{p_2}$.
In particular, $I_{\{p_0, p_2\},C}(1)$ would be a subsheaf of $\F$.
Similarly, if $c$ were zero, then $I_{2p_0,C}(1)$ would be a subsheaf of $\F$.
These subsheaves have Hilbert polynomial $3m + 1$, destabilizing $\F$.

Now, applying an automorphism of $\CC_{p_0} \oplus \CC_{p_2}$ , we may assume that $a=1$ and $c=1$.
Denote by $\F(b,d)$ the sheaf corresponding to $(b, d)\in \CC^2$.
To classify such stable sheaves, note that $I_{\{2p_0,p_2\},C}(1) \simeq \OO_C$ is a subsheaf of
$\F(b,d)$. Since the quotient is $\OO_{L_1}$, $\F(b,d)$ fits in the short exact sequence
\[
0 \to \OO_C \to \F(b,d) \to \OO_{L_1} \to 0.
\]
This sequence splits if and only if $(b, d) = (0, 0)$.
In other words, $\F(b,d)$ is stable if and only if $(b, d)$ is in $\CC^2 \setminus \{(0, 0)\}$.
Clearly $\F(b,d)\simeq \F(kb,kd)$ for $k\in \CC^*$.
Thus $\{[\F(b,d)]\}$ forms a $T$-fixed locus isomorphic to $\PP^1$.
\end{example}

%%%%%%%%%%%%%%%%%%%%%%%%%%%%%%%%%%%%%%%%%%%%% section 3

\section{Torus representation of the tangent spaces I}
\label{representation_1}

The tangent space of the moduli space of semi-stable sheaves
at a point corresponding to a sheaf $\F$ is given by $\Ext^1(\F,\F)$.
Consider
\[
\chi(\F,\F)=\sum_{i=0}^2(-1)^i \Ext^i(\F,\F).
\]
For a stable sheaf $\F$ in $\M_{\PP^2} (4,1)$, we have $\Ext^2(\F,\F)=0$,
and the $T$-action on $\Hom(\F,\F)\simeq \CC$ is trivial.
Hence, in the representation ring of the torus $T$, we have
\[
\Ext^1(\F,\F)= 1- \chi(\F,\F).
\]
Thus, it is enough to compute the representation of $\chi(\F,\F)$.
We use the technique of \cite{mnop1}. By the local-to-global spectral sequence, we have
\[
\chi(\F,\F) =\sum_{i,j=0}^2(-1)^{i+j} \HH^i({\mathcal Ext}^j(\F,\F)).
\]

For $\alpha=0,1,2$, let $U_\alpha$ be the affine open subset of $\PP^2$ defined in Section \ref{fixed_locus_1}.
Let $U_{\alpha\beta} = U_\alpha \cap U_\beta$.
We replace the cohomology with the \v{C}ech complex $\mathfrak{C}^i({\mathcal Ext}^j(\F,\F))$
with respect to the open cover $\{U_\alpha\}$.
A $T$-fixed sheaf $\F$ is necessarily supported on $T$-invariant lines.
Since no $T$-invariant line meets the intersection of three open sets,
we only need to consider $\mathfrak{C}^0$ and $\mathfrak{C}^1$. Thus,
\[
\chi(\F,\F)= \bigoplus_{\alpha=0}^2 \sum_j (-1)^j \Gamma(U_\alpha, {\mathcal Ext}^j(\F,\F))
-\bigoplus_{\alpha,\beta} \sum_j (-1)^j \Gamma(U_{\alpha \beta}, {\mathcal Ext}^j(\F,\F)).
\]

Let $Q_\alpha$ be the $T$-character of $\Gamma(U_\alpha,\F)$. Define
\[
\overline{Q}_\alpha(t_1,t_2)=Q_\alpha(t_1^{-1},t_2^{-1}).
\]

Recall that $R_\alpha$ is the coordinate ring $\Gamma(U_\alpha)$ and $v_\alpha$, $w_\alpha$,
are $T$-characters for the generators of $R_\alpha$.

Consider a $T$-equivariant free resolution of $F_\alpha=\Gamma(U_\alpha,\F)$.
\begin{equation}
  0\to F_s \to \cdots\to F_2\to F_1\to F_\alpha \to 0.
  \label{freeresolution}
\end{equation}
Each term in \eqref{freeresolution} is of the form
\[
F_i =\oplus_j R_\alpha(d_{ij}), \hspace{2em} d_{ij}\in \ZZ^2.
\]
Let
\[
P_\alpha(t_1,t_2)= \sum_{i,j} (-1)^i t^{d_{ij}}.
\]
Then, from the exact sequence \eqref{freeresolution},
\[
Q_\alpha(t_1,t_2)= \frac{P_\alpha(t_1,t_2)}{(1-v_\alpha)(1-w_\alpha)},
\]

The representation $\chi(F_\alpha,F_\alpha)$ is given by the alternating sum
\begin{align*}
\chi(F_\alpha,F_\alpha)&= \sum_{i,j,k,l}(-1)^{i+k}\Hom(R_\alpha(d_{ij}),  R_\alpha(d_{kl}))\\
&= \sum_{i,j,k,l}(-1)^{i+k}R_\alpha(d_{kl}-d_{ij})\\
&= \frac{P_\alpha(t_1,t_2)P_\alpha(t_1^{-1},t_2^{-1})}{(1-v_\alpha)(1-w_\alpha)}\\
&= Q_\alpha \overline{Q}_\alpha(1-v_\alpha^{-1})(1-w_\alpha^{-1}).
\end{align*}

Similarly for intersection $U_{\alpha\beta}$, let $R_{\alpha\beta}$ denote the coordinate ring $\Gamma(U_{\alpha\beta})$.
Let $v_{\alpha\beta}$, $w_{\alpha\beta}$ be $T$-characters for the generators of $R_\alpha$,
where $v_{\alpha\beta}^{-1}$ is in $R_{\alpha\beta}$. For example,
since $R_0=\CC[x,y]$ and $R_{01}=\CC[x,x^{-1},y]$, we take $(v_{01},w_{01})=(t_1, t_2)$.
Similarly, $(v_{12},w_{12})=(t_1^{-1}t_2, t_1^{-1})$, etc.
Then, the $T$-character of $F_{\alpha\beta}=\Gamma(U_{\alpha\beta},\F)$ has an overall factor
\[
\delta(v_{\alpha\beta})=\displaystyle{\sum_{n=-\infty}^{\infty}
v_{\alpha\beta}^n=\frac{1}{1-v_{\alpha\beta}}+\frac{v_{\alpha\beta}^{-1}}{1-v_{\alpha\beta}^{-1}}}.
\]
Let $Q_{\alpha\beta}$ be such that the $T$-character of $F_{\alpha\beta}$ is
\[
\delta(v_{\alpha\beta})Q_{\alpha\beta}.
\]
By the same computation as before, we get
\[
\chi(F_{\alpha\beta},F_{\alpha\beta})= \delta(v_{\alpha\beta}) Q_{\alpha\beta} \overline{Q}_{\alpha\beta}(1-w_{\alpha\beta}^{-1}).
\]

Therefore, we get the following.
\begin{proposition}\label{prop:trep}
\[
\chi(\F,\F)=\sum_{\alpha=0}^{2} Q_\alpha \overline{Q}_\alpha(1-v_\alpha^{-1})(1-w_\alpha^{-1})
-\sum_{\alpha,\beta}\delta(v_{\alpha\beta}) Q_{\alpha\beta} \overline{Q}_{\alpha\beta}(1-w_{\alpha\beta}^{-1}).
\]
\end{proposition}

Although each term in the summation is infinite dimensional, the total sum is necessarily finite, cf. \cite{mnop1}.

Our proof of Theorem \ref{main_theorem} is a case by case computation using the classification in
Section \ref{fixed_locus_1} and Proposition \ref{prop:trep}.
In Example \ref{ex:trep}, we carry out the computation for the last example in
Section \ref{fixed_locus_1}.
The computation for the other sheaves is similar and equally complicated.

By \cite[Lemma 4.1]{carrell},
once the $T$-representation is computed, we can take a generic one-parameter
subgroup of $T$ to compute the numbers $p(i)$ from Theorem \ref{homology_basis}.
For this, we will use the one-parameter subgroup
\begin{equation}
\label{eq:1parmsg}
\lambda(t)=(t, t^l),
\end{equation}
for a sufficiently large $l$.
Thus, $p(i)$ is equal to the number of weights $t_1^a t_2^b$ in the $T$-representation satisfying
\[
b>0 \quad \text{or} \quad (a>0 \quad \text{and} \quad b=0).
\]

\begin{example}\label{ex:trep}
We compute the $T$-representation of the tangent space at the sheaf in example (m) in
Section \ref{fixed_locus_1}. We include the diagram again in Figure \ref{fig:lastex}.
\begin{figure}
\begin{center}
\begin{pspicture}(0,0.5)(5,4.5)
\rput(0.5,0){
\psellipse[fillstyle=solid,fillcolor=lightgray,linewidth=.3mm](.5,1)(3.5,3)
\psellipse[fillstyle=solid,fillcolor=white,linewidth=.3mm](.5,1)(3,2.5)
\pspolygon[fillstyle=solid,fillcolor=white,linestyle=none](.5,1)(.5,4)(-5,4)(-5,-5)(5,-5)(5,1)(.5,1)
\pspolygon[fillstyle=solid,fillcolor=lightgray,linewidth=.3mm](0,0.5)(4,.5)(4,1)(0,1)(0,0.5)
\pspolygon[fillstyle=solid,fillcolor=lightgray,linewidth=.3mm](0,0.5)(.5,.5)(.5,4)(0,4)(0,0.5)
\rput(0,.5){\onebox{2}}\rput(3.5,.5){\onebox{2}}\rput(0,3.5){\onebox{1}}
\rput(2.5,3.2){$1$}
\rput(2,.75){$2$}
\rput(.25,2.5){$1$}
%coordinate axis
\psline[linestyle=dashed]{->}(-.5,.5)(1,.5)\psline[linestyle=dashed]{->}(4.5,.5)(3,.5)
\psline[linestyle=dashed]{->}(0,0)(0,2)
\psline[linestyle=dashed]{->}(4,0)(4,2)
%\rput(2.3,2.7){$U_{\sigma_1}$}\rput(3.8,2.7){$U_{\sigma_2}$}
\rput(1,0.2){$^{(1,0)}$}\rput(-0.4,1.7){$^{(0,1)}$}
\rput(3,0.2){$^{(-1,0)}$}\rput(4.5,1.7){$^{(-1,1)}$}
\psline[linestyle=dashed]{->}(-0.5,4)(1,4)\psline[linestyle=dashed]{->}(0,4.5)(0,3)
\rput(1.5,4.2){$^{(1,-1)}$}\rput(-0.5,3.2){$^{(0,-1)}$}
}
\psline{->}(0.75,.9)(.75,1.25)
\psline{->}(4.25,.9)(4.25,1.25)
\end{pspicture}
\end{center}
\caption{The example (m) in Section \ref{fixed_locus_1}}\label{fig:lastex}
\end{figure}
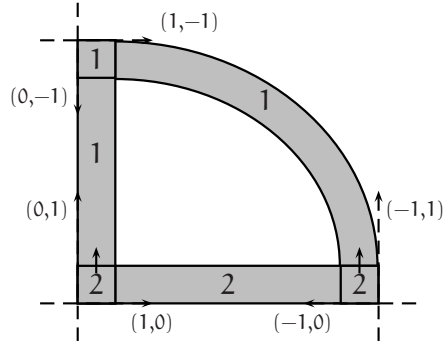
From the diagram, we can see that

\begin{align*}
  Q_0&= 2(1+t_1+t_1^2+\cdots)+(t_2+t_2^2+\cdots)=\frac{2}{1-t_1}+\frac{t_2}{1-t_2}, \\
  Q_1&= 2(1+t_1^{-1}+t_1^{-2}+\cdots)+(t_1^{-1}t_2+t_1^{-2}t_2^2+\cdots)=\frac{2}{1-t_1^{-1}}+\frac{t_1^{-1}t_2}{1-t_1^{-1}t_2},  \\
  Q_2&= (1+t_2^{-1}+t_2^{-2}+\cdots)+(t_1t_2^{-1}+t_1^2t_2^{-2}+\cdots)=\frac{1}{1-t_2^{-1}}+\frac{t_1t_2^{-1}}{1-t_1t_2^{-1}},
\end{align*}

\begin{align*}
  Q_{01}&=2 , \\
  Q_{12}&= 1 ,\\
  Q_{20}&= 1 .
\end{align*}
By applying Proposition \ref{prop:trep}, we get the weight decomposition
\[
\Ext^1(\F,\F)= t_1^{-1}t_2^2+t_2+t_1^{-1}t_2+t_1+t_1^{-1}+2 t_1^2t_2^{-1}+4 t_1t_2^{-1}+4t_2^{-1}+2t_1^{-1}t_2^{-1}.
\]
With respect to the one-parameter subgroup \eqref{eq:1parmsg},
the dimension of the plus cell attached to this fixed point is $4$.
The computations for the other fixed points are similar.
By Theorem \ref{homology_basis}, we get Theorem \ref{main_theorem}.
\end{example}

%%%%%%%%%%%%%%%%%%%%%%%%%%%%%%%%%%%% section 5

\section{Torus fixed locus II}
\label{fixed_locus_2}

In the last two sections we will present a different method for investigating the torus fixed locus.
We change slightly the conventions. In the sequel $T$ will be the torus $(\CC^*)^3/\{ (a,a,a) \mid a \in \CC^* \}$
acting on $\PP^2$ by
\[
(t_0, t_1, t_2) \cdot (x,y,z) = (t_0^{-1} x, t_1^{-1} y, t_2^{-1} z).
\]
This is, of course, equivalent to the set-up in Sections \ref{fixed_locus_1} and \ref{representation_1}.
We fix a vector space $V \simeq \CC^3$ and we make the identification $\PP^2 = \PP(V)$.
We fix a basis $\{X, Y, Z \}$ of $V^*$ dual to the standard basis of $V$.
The induced action of $T$ on the symmetric algebra of $V^*$ is given by
$t \cdot X^iY^jZ^k = t_0^i t_1^j t_2^k X^i Y^j Z^k$.

\noindent \\
We begin by recalling some notations and results from \cite{drezet-maican}.

\noindent
\begin{tabular}{rcl}
$\M_0$ & = & the open subset of $\M$ consisting of sheaves $\F$ having resolution
\\
& & $ 0 \to 3\OO(-2) \stackrel{\varphi}{\to} 2\OO(-1) \oplus \OO \to \F \to 0$,
\\
& & where $\varphi_{11}$ has linearly independent maximal minors;
\\
$\M_{01}$ & = & the locally closed subset of $\M_0$ given by the condition that
\\
& & the maximal minors of $\varphi_{11}$ have a common linear factor;
\\
$W_{\phantom{0}}$ & = & $\Hom(3\OO(-2), 2\OO(-1) \oplus \OO)$;
\\
$W_0$ & = & the set of injective morphisms $\varphi \in W$ such that
\\
& & the maximal minors of $\varphi_{11}$ are linearly independent;
\\
$G_0$ & = & $\big( \Aut(3\OO(-2)) \times \Aut(2\OO(-1) \oplus \OO) \big)/\CC^*$,
\\
& & where $\CC^*$ is the subgroup of homotheties;
\\
$\M_1$ & = & the smooth subvariety of codimension $2$ consisting of
sheaves
\\
& & $\F$ having resolution $0 \to \OO(-3) \oplus \OO(-1) \stackrel{\varphi}{\to} 2\OO \to \F \to 0$,
\\
& & where $\varphi_{12}$ has linearly independent entries;
\\
$W_1$ & = & set of morphisms $\varphi \colon \OO(-3) \oplus \OO(-1) \to 2\OO$ as above;
\\
$G_1$ & = & $\big( \Aut(\OO(-3) \oplus \OO(-1)) \times \Aut(2\OO) \big)/\CC^*$.
\end{tabular}

\noindent \\
The moduli space $\M_{\PP^2}(4,1)$ is the disjoint union of $\M_0$ and $\M_1$.
Moreover, $\M_0$ is isomorphic to the geometric quotient $W_0/G_0$, while $\M_1$ is isomorphic
to the geometric quotient $W_1/G_1$.
The sheaves in $\M_0 \setminus \M_{01}$ are precisely the non-split extensions
\[
0 \to \OO_Q \to \F \to \OO_S \to 0,
\]
where $Q \subset \PP^2$ is a quartic curve and $S \subset \PP^2$ is a zero-dimensional scheme
of length $3$ that is not contained in a line.
Here $Q = \{ \det(\varphi) = 0 \}$ and $S$ has ideal generated by the maximal minors of $\varphi_{11}$.
The sheaves in $\M_{01}$ are precisely the extension sheaves
\[
0 \to \OO_C \to \F \to \OO_L \to 0,
\]
satisfying $\HH^1(\F) = 0$, where $C$ is a cubic curve and $L$ a line. In fact $L$ has equation $l = 0$,
where $l$ is the common factor of the maximal minors of $\varphi_{11}$.
We have $\dim(\Ext^1(\OO_L, \OO_C))=3$.
The sheaves in $\M_1$ are precisely the kernels of surjective morphisms $\OO_Q(1) \to \CC_p$,
where $Q$ is a quartic curve and $\CC_p$ is the structure sheaf of a closed point $p$.

\noindent \\
The $T$-fixed points in $\PP^2$ are $p_0 = (1, 0, 0)$, $p_1 = (0,1,0)$, $p_2 = (0,0,1)$.
The fixed double points are $p_{01}$, $p_{02}$, $p_{10}$, $p_{12}$, $p_{20}$, $p_{21}$,
where $p_{ij}$ is supported on $p_i$ and is contained in the line $p_i p_j$.
Denote by $s_0$, $s_1$, $s_2$ the fixed triple points supported on $p_0$, $p_1$, $p_2$,
that are not contained in a line.
The subsets $\M_0$, $\M_1$, $\M_{01}$ are $T$-invariant.

\subsection{Fixed points in $\M_0 \setminus \M_{01}$}  %%%%%%%% subsection 5.1
\label{5.1}

Since $\HH^0(\F)$ generates $\OO_Q$ we see that $\F$ is fixed by $T$ if and only if $Q$ and $S$
are fixed by $T$. Thus $Q = \{ X^i Y^j Z^k = 0\}$, where $i+j+k=4$, and $S$ belongs to the following
list: $\{ p_0, p_1, p_2 \}$, $\{ p_{01}, p_2 \}$, $\{ p_{02}, p_1 \}$, $\{ p_{10}, p_2 \}$, $\{ p_{12}, p_0 \}$,
$\{ p_{20}, p_1 \}$, $\{ p_{21}, p_0 \}$, $\{ s_0 \}$, $\{ s_1 \}$, $\{ s_2 \}$.
For each $S$ there are precisely $12$ invariant quartics containing $S$, so we have $120$
fixed points in $\M_0 \setminus \M_{01}$.

\noindent \\
Assume $S = \{ p_0, p_1, p_2 \}$.
Then $I_S = (XY, XZ, YZ)$ and $\F = {\mathcal Coker}(\varphi)$, where $\varphi$ is represented
by one of the following matrices ($i+j+k =2$, $i, j, k \ge 0$):
\[
\label{5.1.1}
\tag{5.1.1}
\left[
\begin{array}{ccc}
Y & 0 & X \\
0 & Z & X \\
X^i Y^j Z^k & 0 & 0
\end{array}
\right], \quad \left[
\begin{array}{ccc}
Y & 0 & X \\
0 & Z & X \\
0 & X^i Y^j & 0
\end{array}
\right], \quad \left[
\begin{array}{ccc}
Y & 0 & X \\
0 & Z & X \\
0 & 0 & Y^j Z^k
\end{array}
\right].
\]
Assume that $S = \{ p_{21}, p_0\}$.
Then $I_S = (XY, XZ, Y^2)$ and $\F = {\mathcal Coker}(\varphi)$, where $\varphi$ is one of the following morphisms:
\[
\label{5.1.2}
\tag{5.1.2}
\left[
\begin{array}{ccc}
X & 0 & Y \\
0 & Y & Z \\
0 & 0 & X^iY^jZ^k
\end{array}
\right], \quad \left[
\begin{array}{ccc}
X & 0 & Y \\
0 & Y & Z \\
0 & X^iZ^k & 0
\end{array}
\right], \quad \left[
\begin{array}{ccc}
X & 0 & Y \\
0 & Y & Z \\
Y^jZ^k & 0 & 0
\end{array}
\right].
\]
To obtain the fixed points corresponding to the other schemes of length $3$ that are a union
of a double point and a closed point apply permutations in the variables $X, Y, Z$
to the above matrices.

\noindent \\
Assume that $S= \{ s_0\}$. Then $I_S = (YZ, Y^2, Z^2)$ and $\F = {\mathcal Coker}(\varphi)$, where $\varphi$
belongs to the list:
\[
\label{5.1.3}
\tag{5.1.3}
\left[
\begin{array}{ccc}
Y & 0 & Z \\
0 & Z & Y \\
0 & 0 & X^i Y^j Z^k
\end{array}
\right], \quad \left[
\begin{array}{ccc}
Y & 0 & Z \\
0 & Z & Y \\
0 & X^i Y^j & 0
\end{array}
\right], \quad \left[
\begin{array}{ccc}
Y & 0 & Z \\
0 & Z & Y \\
X^i Z^k & 0 & 0
\end{array}
\right].
\]
To obtain the fixed points corresponding to $S = \{ s_1\}$ swap $X$ and $Y$ in the above
matrices. For $S = \{ s_2 \}$ swap $X$ and $Z$.

\subsection{Fixed points in $\M_{01}$} %%%%%%%%%%%%%% subsection 5.2
\label{5.2}

Denote $\M_{L,C} = \PP(\Ext^1(\OO_L, \OO_C)) \cap \M_{01}$. Assume that $\F \in \M_{L,C}$
is fixed by $T$. Since $\HH^0(\F)$ generates $\OO_C$ we see that $L$ and $C$ are fixed
by $T$. Assume that $L = \{ X=0 \}$ and $C = \{ g = 0 \}$, where $g = X^i Y^j Z^k$, $i+j+k =3$.
Fix $q_1, q_2, q_3 \in \Sym^2 V^*$ such that $q_1 Y + q_2 Z + q_3 X = g$.
Consider the set $U \subset W_0$ of morphisms of the form
\[
\left[
\begin{array}{ccc}
X & 0 & -Y \\
0 & X & -Z \\
q_1 - uZ & q_2 + uY & q_3
\end{array}
\right], \qquad u = bY+cZ, \quad b, c \in \CC.
\]
Let $\rho \colon W_0 \to \M_0$, $\rho(\varphi) = [{\mathcal Coker}(\varphi)]$, be the geometric quotient map.

\begin{proposition}
The restricted map $\rho \colon U \to \M_{L,C}$ is an isomorphism.
\end{proposition}

\begin{proof}
Clearly each sheaf $\F$ giving a point in $\M_{L,C}$ is isomorphic to the cokernel of a morphism
\[
\varphi = \left[
\begin{array}{ccc}
X & 0 & -Y \\
0 & X & -Z \\
q_1' & q_2' & q_3'
\end{array}
\right],
\]
where $q_1'Y + q_2' Z + q_3' X = g$.
Since $(q_1' - q_1) Y + (q_2' - q_2) Z + (q_3'-q_3) X=0$,
there are $u_1, u_2, u_3 \in V^*$ such that
\[
\left[
\begin{array}{ccc}
q_1' & q_2' & q_3'
\end{array}
\right] - \left[
\begin{array}{ccc}
q_1 & q_2 & q_3
\end{array}
\right] = \left[
\begin{array}{ccc}
u_1 & u_2 & u_3
\end{array}
\right] \left[
\begin{array}{ccc}
X & 0 & -Y \\
0 & X & -Z \\
-Z & Y & 0
\end{array}
\right].
\]
Thus, writing $u_3 = aX + bY + cZ$, $u= bY + cZ$, $a, b, c \in \CC$, we have the equivalences
\[
\varphi \sim \left[
\begin{array}{ccc}
X & 0 & -Y \\
0 & X & -Z \\
q_1 - u_3 Z & q_2 + u_3 Y & q_3
\end{array}
\right] \sim \left[
\begin{array}{ccc}
X & 0 & -Y \\
0 & X & -Z \\
q_1 - u Z & q_2 + u Y & q_3
\end{array}
\right].
\]
This shows that $\rho$ is surjective.
It remains to prove injectivity. Assume that for $u = bY + cZ$, $u' = b'Y+c'Z$, we have
\[
\left[
\begin{array}{ccc}
X & 0 & -Y \\
0 & X & -Z \\
q_1 - u Z & q_2 + u Y & q_3
\end{array}
\right] \sim \left[
\begin{array}{ccc}
X & 0 & -Y \\
0 & X & -Z \\
q_1 - u' Z & q_2 + u' Y & q_3
\end{array}
\right].
\]
The stabilizer of $\varphi_{11}$ in $(\GL(3, \CC) \times \GL(2, \CC))/\CC^*$ is trivial because
$\varphi_{11}$ is stable as a Kronecker module.
It follows that there are $a \in \CC^*$, $v_1, v_2 \in V^*$ such that
\[
\left[
\begin{array}{ccc}
v_1 & v_2 & a
\end{array}
\right] \left[
\begin{array}{ccc}
X & 0 & -Y \\
0 & X & -Z \\
q_1 -u Z & q_2 + u Y & q_3
\end{array}
\right] = \left[
\begin{array}{ccc}
q_1 - u' Z & q_2 + u' Y & q_3
\end{array}
\right].
\]
Thus
\[
(1-a) \left[
\begin{array}{ccc}
q_1 & q_2 & q_3
\end{array}
\right] = \left[
\begin{array}{ccc}
v_1 & v_2 & a u - u'
\end{array}
\right] \left[
\begin{array}{ccc}
X & 0 & -Y \\
0 & X & -Z \\
-Z & Y & 0
\end{array}
\right],
\]
therefore $(1-a) (q_1 Y + q_2 Z + q_3 X) =0$, hence $a=1$ and
\[
\left[
\begin{array}{ccc}
v_1 & v_2 & u - u'
\end{array}
\right] \left[
\begin{array}{ccc}
X & 0 & -Y \\
0 & X & -Z \\
-Z & Y & 0
\end{array}
\right] =0.
\]
It follows that $(v_1, v_2, u - u')$ is a multiple of $(Z, -Y, X)$ forcing $u = u'$.
\end{proof}

\noindent
We now describe the induced action of $T$ on $U$.
We assume that precisely one among $q_1, q_2, q_3$ is non-zero.
Thus
\[
t q_1 = t_0^i t_1^{j-1} t_2^k q_1, \qquad
t q_2 = t_0^i t_1^j t_2^{k-1} q_2, \qquad
t q_3 = t_0^{i-1} t_1^j t_2^k q_3,
\]
\begin{alignat*}{10}
t & \left[
\begin{array}{ccc}
X & 0 & -Y \\
0 & X & -Z \\
q_1 - u Z & q_2 + u Y & q_3
\end{array}
\right]
\\
& = \left[
\begin{array}{lll}
t_0 X & 0 & -t_1 Y \\
0 & t_0 X & -t_2 Z \\
t_0^i t_1^{j-1} t_2^k q_1 - (t u) t_2 Z & t_0^i t_1^j t_2^{k-1} q_2 + (t u) t_1 Y & \phantom{-} t_0^{i-1} t_1^j t_2^k q_3
\end{array}
\right]
\\
& \sim \left[
\begin{array}{lll}
X & 0 & -Y \\
0 & X & -Z \\
t_0^{i-1} t_1^j t_2^k q_1 - t_0^{-1} t_1^{} t_2^{} (t u) Z & t_0^{i-1} t_1^j t_2^k q_2 + t_0^{-1} t_1^{} t_2^{} (t u) Y
& \phantom{-} t_0^{i-1} t_1^j t_2^k q_3
\end{array}
\right]
\\
& \sim \left[
\begin{array}{lll}
X & 0 & -Y \\
0 & X & -Z \\
q_1 - t_0^{-i} t_1^{1-j} t_2^{1-k} (t u) Z & q_2 + t_0^{-i} t_1^{1-j} t_2^{1-k} (t u) Y & \phantom{-} q_3
\end{array}
\right].
\end{alignat*}
Identify $U$ with the affine plane with coordinates $(b, c)$.
The action of $T$ on $U$ is given by
\[
t (b,c) = (t_0^{-i} t_1^{2-j} t_2^{1-k} b,\  t_0^{-i} t_1^{1-j} t_2^{2-k} c).
\]
If $i \neq 0$, or if $i=0$ and $(j,k) \neq (2,1), (1,2)$, we get only one fixed point, namely $(0,0)$.
If $(i,j,k)= (0,2,1)$ we get the fixed line $(b,0)$, $b \in \CC$.
If $(i,j,k)= (0,1,2)$ we get the fixed line $(0,c)$, $c \in \CC$.
Thus the isolated fixed points in $\M_{L,C}$ are of the form $\F = {\mathcal Coker}(\varphi)$, where $\varphi$
is precisely one of the following morphisms:
\[
\label{5.2.1}
\tag{5.2.1}
\left[
\begin{array}{ccc}
X & 0 & -Y \\
0 & X & -Z \\
0 & 0 & X^{i-1} Y^j Z^k
\end{array}
\right], \quad \left[
\begin{array}{ccc}
X & 0 & -Y \\
0 & X & -Z \\
Y^2 & 0 & 0
\end{array}
\right], \quad \left[
\begin{array}{ccc}
X & 0 & -Y \\
0 & X & -Z \\
0 & Z^2 & 0
\end{array}
\right].
\]
The first morphism corresponds to the case when $i \neq 0$,
the second morphism to the case $(i,j,k) = (0, 3, 0)$,
the third morphism to the case $(i,j,k) = (0, 0, 3)$.
In the case $(i,j,k) = (0,2,1)$, respectively $(i,j,k) = (0,1,2)$, we get the fixed lines
$\F = {\mathcal Coker}(\varphi)$, where
\[
\label{5.2.2}
\tag{5.2.2}
\varphi = \left[
\begin{array}{ccc}
X & 0 & -Y \\
0 & X & -Z \\
(1-b) YZ & bY^2 & 0
\end{array}
\right], \quad \varphi = \left[
\begin{array}{ccc}
X & 0 & -Y \\
0 & X & -Z \\
(1-c)Z^2 & cYZ & 0
\end{array}
\right],
\]
respectively, $b \in \CC$, $c \in \CC$.
To obtain the fixed points in $\M_{\{Y=0\}, C}$ interchange $X$ and $Y$ in \ref{5.2.1} and \ref{5.2.2}.
To obtain the fixed points in $\M_{\{Z=0\}, C}$ interchange $X$ and $Z$.
In conclusion there are $24$ isolated fixed points in $\M_{01}$ and $6$ fixed affine lines.

\subsection{Fixed points in $\M_1$}  %%%%%%%%%%%%%%% subsection 5.3
\label{5.3}

The kernel $\F$ of a surjective morphism $\OO_Q(1) \to \CC_p$ is fixed by $T$ precisely if $Q$
is a fixed quartic and $p \in \{ p_0, p_1, p_2 \}$.
If $p = p_2$, then $\F = {\mathcal Coker}(\varphi)$, where $\varphi$ belongs to the list:
\[
\left[
\begin{array}{cc}
X^i Y^j Z^k & X \\
0 & Y
\end{array}
\right], \quad i+j+k = 3, \qquad \left[
\begin{array}{cc}
0 & X \\
X^i Z^k & Y
\end{array}
\right], \quad i+k = 3.
\]
Thus there are $42$ isolated fixed points for the action of $T$ on $\M_1$.

\begin{proposition}
Consider the line $L = \{ X = 0 \}$, the cubic curve $C = \{Y^2 Z = 0 \}$ and the quartic curve
$Q = \{ X Y^2 Z =0 \}$ in $\PP^2$.
Let $\Lambda = \Lambda_{L,C}$ be the affine line of torus fixed points in $\M_{L,C}$.
Then $\overline{\Lambda} \simeq \PP^1$ and $\overline{\Lambda} \setminus \Lambda$ is the kernel of the morphism
$\OO_Q(1) \to \CC_{p_2}$.
\end{proposition}

\begin{proof}
Clearly $\PP(\Ext^1(\OO_L, \OO_C)) \cap \M_1$ is a projective line,
being the set of isomorphism classes of kernels of surjective morphisms $\OO_Q(1) \to \CC_p$, $p \in L$.
Recall that $\M_{L,C} = \PP(\Ext^1(\OO_L, \OO_C)) \cap \M_0 = \PP^2 \setminus \PP^1$,
hence $\overline{\Lambda}$ is the closure of $\Lambda$ in $ \PP(\Ext^1(\OO_L, \OO_C))$,
so it is isomorphic to $\PP^1$ and $\overline{\Lambda} \setminus \Lambda$ is the kernel of
a morphism $\OO_Q(1) \to \CC_p$ for a $T$-fixed point $p = p(L,C) \in L$.

If we interchange $X$ and $Z$ the fixed line $\Lambda_{X, Y^2 Z}$ becomes $\Lambda_{Z, X Y^2}$,
hence $p(L,C)=p(X, Y^2 Z)$ becomes $p(Z, X Y^2)$, which shows that these two points lie on the line
$\{ Y = 0 \}$. We conclude that $p(L,C)=p_2$.
\end{proof}

\noindent
In conclusion, we have proved the first part of Theorem \ref{main_theorem} concerning $\M^T$.
Denote by $\Lambda_1, \ldots, \Lambda_6$ the projective lines of $T$-fixed points in $\M$.

%%%%%%%%%%%%%%%%%%%%%%%%%%%%%%%%%%%%%%%%%%%%% section 6

\section{Torus representation of the tangent spaces II}
\label{representation_2}

\noindent
For convenience we introduce the following notations:

\begin{tabular}{rcl}
$\chi_0$ & = & the trivial character of $T$;
\\
$\sigma^0$ & = & $\{ t_0^{-1} t_1^{}, t_0^{-1} t_2^{}, t_0^{} t_1^{-1}, t_1^{-1} t_2^{}, t_0^{} t_2^{-1}, t_1^{} t_2^{-1} \}$;
\\
$\sigma^l$ & = & $\{ t_0^i t_1^j t_2^k, \quad i+j+k = l, \quad i,j,k \ge 0 \}$ for $l \ge 1$;
\\
$\sigma^l_{ijk}$ & = & $\sigma^l \setminus \{ t_0^i t_1^j t_2^k \}$;
\\
$\sigma^l_i$ & = & $\sigma^l \setminus t_i \sigma^{l-1}$ for $l \ge 2$ and $i = 0, 1, 2$.
\end{tabular}

\subsection{Tangent spaces at fixed points in $\M_0$}     %%%%%%% subsection 6.1
\label{6.1}

Let $\rho \colon W_0 \to \M_0$ denote the geometric quotient map, $\rho(\varphi) = [{\mathcal Coker}(\varphi)]$.
Let $\F$ give a fixed point in $\M_0$.
There is $\varphi \in \rho^{-1}([\F])$ such that there are morphisms of groups
\[
\alpha \colon T \to \Aut(3\OO(-2)), \qquad \beta \colon T \to \Aut(2\OO(-1) \oplus \OO),
\]
\[
\alpha = \left[
\begin{array}{ccc}
\alpha_1 & 0 & 0 \\
0 & \alpha_2 & 0 \\
0 & 0 & \alpha_3
\end{array}
\right], \qquad \beta = \left[
\begin{array}{ccc}
\beta_1 & 0 & 0 \\
0 & \beta_2 & 0 \\
0 & 0 & \beta_3
\end{array}
\right],
\]
satisfying the property: $t \varphi = \beta(t) \varphi \alpha(t)$ for all $t \in T$.
The morphisms $\varphi$ from \ref{5.1.1}, \ref{5.1.2}, \ref{5.1.3}, \ref{5.2.1} clearly satisfy the required property.
For the morphism
\[
\varphi = \left[
\begin{array}{ccc}
X & 0 & -Y \\
0 & X & -Z \\
(1-b) YZ & bY^2 & 0
\end{array}
\right]
\]
we have
\begin{multline*}
t \varphi = \left[
\begin{array}{ccc}
t_0 X & 0 & -t_1 Y \\
0 & t_0 X & -t_2 Z \\
t_1 t_2 (1-b) YZ & t_1^2 bY^2 & 0
\end{array}
\right] =
\\
\left[
\begin{array}{ccc}
t_1 & 0 & 0 \\
0 & t_2 & 0 \\
0 & 0 & t_0^{-1} t_1^2 t_2
\end{array}
\right] \varphi \left[
\begin{array}{ccc}
t_0 t_1^{-1} & 0 & 0 \\
0 & t_0 t_2^{-1} & 0 \\
0 & 0 & 1
\end{array}
\right].
\end{multline*}
This checks the property for the morphisms at \ref{5.2.2}.
Fix $\F$ and $\varphi$ as above.
Denote
\[
{\mathfrak g}_0 =  \T_1 G_0 = \big( \End(3\OO(-2)) \oplus \End(2\OO(-1) \oplus \OO) \big)/\CC.
\]
The map $G_0 \to W_0$, given by $(g,h) \mapsto h \varphi g^{-1}$, has differential at $1$
\[
{\mathfrak g} \to \T_{\varphi} W_0 = W \quad \text{given by} \quad (A, B) \mapsto B \varphi - \varphi A.
\]
We identify ${\mathfrak g}$ with its image and we identify $\T_{[\F]} \M_{\PP^2}(4,1)$ with $W/{\mathfrak g}$.
From the commutative diagram
\[
\xymatrix
{
W_0 \ar[r]^-{\mu_t} \ar[d]_-{\rho} & W_0 \ar[d]^-{\rho} \\
\M_0 \ar[r]^-{\mu_t} & \M_0
}
\]
we have the relations (in which $w \in W$ and $\theta$ denotes the map $\psi \mapsto \beta(t) \psi \alpha(t)$)
\[
\diff (\mu_t)_{[\F]} (\diff \rho_{\varphi}(w)) = \diff \rho_{t \varphi}(\diff (\mu_t)_{\varphi}(w)) =
\diff \rho_{t \varphi}(t w) = \diff \rho_{\theta (\varphi)}(t w).
\]
Since $\rho \circ \theta = \rho$, we see that $\diff \rho_{\varphi} (w) =
\diff \rho_{\theta (\varphi)}(\diff \theta_{\varphi}(w)) = \diff \rho_{\theta(\varphi)}(\theta(w))$.
It follows that
\[
\diff \rho_{\theta (\varphi)}(t w) = \diff \rho_{\varphi} (\theta^{-1}(t w)) =
\diff \rho_{\varphi} (\beta(t)^{-1} (t w) \alpha(t)^{-1}).
\]
Let $[w]$ denote the class in $\T_{[\F]} \M$ of $w$.
The above calculations show that the action of $T$ on $\T_{[\F]} \M$, denoted by $\star$,
is given by the formula
\[
\label{6.1.1}
\tag{6.1.1}
t \star [w] = [ \beta(t)^{-1} (t w) \alpha(t)^{-1}]
\]
and is induced by an action of $T$ on $W$ given by the same formula.
We represent in the following array the weights for the action of $T$ on $W$.
We adopt the following convention: each character appears as many times as is the dimension
of its eigenspace.
\[
\label{6.1.2}
\tag{6.1.2}
\left[
\begin{array}{rrr}
\beta_1^{-1} \alpha_1^{-1} \sigma^1 & \beta_1^{-1} \alpha_2^{-1} \sigma^1 & \beta_1^{-1} \alpha_3^{-1} \sigma^1 \\
\beta_2^{-1} \alpha_1^{-1} \sigma^1 & \beta_2^{-1} \alpha_2^{-1} \sigma^1 & \beta_2^{-1} \alpha_3^{-1} \sigma^1 \\
\beta_3^{-1} \alpha_1^{-1} \sigma^2 & \beta_3^{-1} \alpha_2^{-1} \sigma^2 & \beta_3^{-1} \alpha_3^{-1} \sigma^2
\end{array}
\right].
\]
The subspace ${\mathfrak g} \subset W$ is invariant because
\begin{alignat*}{10}
t \star (B \varphi - \varphi A) & = \beta(t)^{-1} ((t B) (t \varphi) - (t \varphi) (t A)) \alpha(t)^{-1} \\
& = \beta(t)^{-1} (t B) \beta(t) \varphi - \varphi \alpha(t) (t A) \alpha(t)^{-1}.
\end{alignat*}
The induced action on ${\mathfrak g}$ is given by the formula
\[
\label{6.1.3}
\tag{6.1.3}
t \star (A, B) = (\alpha(t) (t A) \alpha(t)^{-1}, \beta(t)^{-1} (t B) \beta(t)).
\]
Its weights are represented in the following array (same convention as above):
\[
\label{6.1.4}
\tag{6.1.4}
\left[
\begin{array}{lllll}
\chi_0^{} & \alpha_1^{} \alpha_2^{-1} & \alpha_1^{} \alpha_3^{-1} & \chi_0^{} & \beta_1^{-1} \beta_2^{} \\
\alpha_2^{} \alpha_1^{-1} & \chi_0^{} & \alpha_2^{} \alpha_3^{-1} & \beta_2^{-1} \beta_1^{} & \chi_0^{} \\
\alpha_3^{} \alpha_1^{-1} & \alpha_3^{} \alpha_2^{-1} & \chi_0^{} & \beta_3^{-1} \beta_1^{} \sigma^1 & \beta_3^{-1} \beta_2^{} \sigma^1
\end{array}
\right].
\]

\begin{example}
Consider the morphism $\varphi$ represented by the first matrix at \ref{5.1.1}. We have
\begin{alignat*}{10}
t \varphi & = \left[
\begin{array}{ccc}
t_1 Y & 0 & t_0 X \\
0 & t_2 Z & t_0 X \\
t_0^i t_1^j t_2^k X^i Y^j Z^k & 0 & 0
\end{array}
\right]
\\
& = \left[
\begin{array}{ccc}
t_0 & 0 & 0 \\
0 & t_0 & 0 \\
0 & 0 & t_0^{i+1} t_1^{j-1} t_2^k
\end{array}
\right] \varphi \left[
\begin{array}{ccc}
t_0^{-1} t_1^{} & 0 & 0 \\
0 & t_0^{-1} t_2^{} & 0 \\
0 & 0 & 1
\end{array}
\right].
\end{alignat*}
Tableaux \ref{6.1.2} and \ref{6.1.4} take the form
\[
\left[
\begin{array}{rrr}
t_1^{-1} \sigma^1 & t_2^{-1} \sigma^1 & t_0^{-1} \sigma^1 \\
t_1^{-1} \sigma^1 & t_2^{-1} \sigma^1 & t_0^{-1} \sigma^1 \\
t_0^{-i} t_1^{-j} t_2^{-k} \sigma^2 & t_0^{-i} t_1^{1-j} t_2^{-1-k} \sigma^2 & t_0^{-1-i} t_1^{1-j} t_2^{-k} \sigma^2
\end{array}
\right],
\]
\[
\left[
\begin{array}{lllll}
\chi_0^{} & t_1^{} t_2^{-1} & t_0^{-1} t_1^{} & \chi_0^{} & \chi_0^{} \\
t_1^{-1} t_2^{} & \chi_0^{} & t_0^{-1} t_2^{} & \chi_0^{} & \chi_0^{} \\
t_0^{} t_1^{-1} & t_0^{} t_2^{-1} & \chi_0^{} & t_0^{-i} t_1^{1-j} t_2^{-k} \sigma^1 & t_0^{-i} t_1^{1-j} t_2^{-k} \sigma^1
\end{array}
\right].
\]
Removing the characters occurring in the second array from the first array we obtain the following
list of characters which describes the weight decomposition of $\T_{[\F]} \M$:
\[
\sigma^0, \quad
\{t_0^{-1-i} t_1^{-j} t_2^{-1-k}(t_0 t_1 \sigma^2_2, \quad t_0 t_2 \sigma^2_1, \quad t_1 t_2 \sigma^2_0, \quad t_0 t_1 t_2 \sigma^1) \}
\setminus \{ \chi_0 \}.
\]
Let $f = X^{i+1} Y^j Z^{k+1}$ be the monomial defining $Q$.
Note that the coefficient in front of the parentheses is $((t f)/f)^{-1}$ and $\beta_3 = (t f)/ t_1 t_2 f$.
For the other two matrices at \ref{5.1.1} we can check that $\beta_3$ is given by the same formula.
Since $\alpha_1$, $\alpha_2$, $\alpha_3$, $\beta_1$, $\beta_2$ are unchanged,
it follows that for all torus fixed points given at \ref{5.1.1}
the weight decomposition of $\T_{[\F]} \M$ is given by the above list in which $t_0^{-1-i} t_1^{-j} t_2^{-1-k}$
gets replaced by $((t f)/f)^{-1}$.
\end{example}

\noindent
In table 1A below we have the weight decompositions for fixed points in $\M_0 \setminus \M_{01}$
obtained by performing similar calculations as in the example above on the matrices at \ref{5.1}.
The ideal of $S$ is given in the first column. The quartic $Q$ is given by the equation $X^i Y^j Z^k =0$.

\begin{table}[!hpt]{Table 1A}
\begin{center}
\begin{tabular}{|c|c|}
\hline
$(XY, XZ, YZ)$
&
\begin{tabular}{l}
$\sigma^0$, \\
$\{t_0^{-i} t_1^{-j} t_2^{-k}
(t_0^{} t_1^{} \sigma^2_2, \quad t_0^{} t_2^{} \sigma^2_1, \quad t_1^{} t_2^{} \sigma^2_0, \quad t_0^{} t_1^{} t_2^{} \sigma^1) \}
\setminus \{ \chi_0^{} \}$
\end{tabular}
\\
\hline
$(XY, XZ, Y^2)$
&
\begin{tabular}{l}
$\sigma^0 \setminus \{ t_1^{} t_2^{-1} \}, \quad t_1^{-2} t_2^2$, \\
$\{t_0^{-i} t_1^{-j} t_2^{-k}
(t_0^{} t_1^{} \sigma^2_2, \quad t_0^{} t_2^{} \sigma^2_1, \quad t_1^2 \sigma^2_0, \quad t_0^{} t_1^{} t_2^{} \sigma^1) \}
\setminus \{ \chi_0^{} \}$
\end{tabular}
\\
\hline
$(YZ, XZ, Y^2)$
&
\begin{tabular}{l}
$\sigma^0 \setminus \{ t_0^{-1} t_1^{} \}, \quad t_0^2 t_1^{-2}$, \\
$\{t_0^{-i} t_1^{-j} t_2^{-k}
(t_1^{} t_2^{} \sigma^2_0, \quad t_0^{} t_2^{} \sigma^2_1, \quad t_1^2 \sigma^2_2, \quad t_0^{} t_1^{} t_2^{} \sigma^1) \}
\setminus \{ \chi_0^{} \}$
\end{tabular}
\\
\hline
$(XY, XZ, Z^2)$
&
\begin{tabular}{l}
$\sigma^0 \setminus \{ t_1^{-1} t_2^{} \}, \quad t_1^2 t_2^{-2}$, \\
$\{t_0^{-i} t_1^{-j} t_2^{-k}
(t_0^{} t_1^{} \sigma^2_2, \quad t_0^{} t_2^{} \sigma^2_1, \quad t_2^2 \sigma^2_0, \quad t_0^{} t_1^{} t_2^{} \sigma^1) \}
\setminus \{ \chi_0^{} \}$
\end{tabular}
\\
\hline
$(XY, YZ, Z^2)$
&
\begin{tabular}{l}
$\sigma^0 \setminus \{ t_0^{-1} t_2^{} \}, \quad t_0^2 t_2^{-2}$, \\
$\{t_0^{-i} t_1^{-j} t_2^{-k}
(t_0^{} t_1^{} \sigma^2_2, \quad t_1^{} t_2^{} \sigma^2_0, \quad t_2^2 \sigma^2_1, \quad t_0^{} t_1^{} t_2^{} \sigma^1) \}
\setminus \{ \chi_0^{} \}$
\end{tabular}
\\
\hline
$(XY, YZ, X^2)$
&
\begin{tabular}{l}
$\sigma^0 \setminus \{ t_0^{} t_2^{-1} \}, \quad t_0^{-2} t_2^2$, \\
$\{t_0^{-i} t_1^{-j} t_2^{-k}
(t_0^{} t_1^{} \sigma^2_2, \quad t_1^{} t_2^{} \sigma^2_0, \quad t_0^2 \sigma^2_1, \quad t_0^{} t_1^{} t_2^{} \sigma^1) \}
\setminus \{ \chi_0^{} \}$
\end{tabular}
\\
\hline
$(XZ, YZ, X^2)$
&
\begin{tabular}{l}
$\sigma^0 \setminus \{ t_0^{} t_1^{-1} \}, \quad t_0^{-2} t_1^2$, \\
$\{t_0^{-i} t_1^{-j} t_2^{-k}
(t_0^{} t_2^{} \sigma^2_1, \quad t_1^{} t_2^{} \sigma^2_0, \quad t_0^2 \sigma^2_2, \quad t_0^{} t_1^{} t_2^{} \sigma^1) \}
\setminus \{ \chi_0^{} \}$
\end{tabular}
\\
\hline
$(YZ, Y^2, Z^2)$
&
\begin{tabular}{l}
$t_0^{} (t_1^{-1},\ t_1^{-1},\ t_2^{-1},\ t_2^{-1},\ t_1^{-2} t_2^{},\ t_1^{} t_2^{-2})$, \\
$\{t_0^{-i} t_1^{-j} t_2^{-k}(t_1^{} t_2^{} \sigma^2, \quad t_1^2 \sigma_2^2, \quad t_2^2 \sigma_1^2) \}
\setminus \{ \chi_0^{} \}$
\end{tabular}
\\
\hline
$(XZ, X^2, Z^2)$
&
\begin{tabular}{l}
$t_1^{} (t_0^{-1},\ t_0^{-1},\ t_2^{-1},\ t_2^{-1},\ t_0^{-2} t_2^{},\ t_0^{} t_2^{-2})$, \\
$\{t_0^{-i} t_1^{-j} t_2^{-k}(t_0^{} t_2^{} \sigma^2, \quad t_0^2 \sigma_2^2, \quad t_2^2 \sigma_0^2) \}
\setminus \{ \chi_0^{} \}$
\end{tabular}
\\
\hline
$(XY, X^2, Y^2)$
&
\begin{tabular}{l}
$t_2^{} (t_0^{-1},\ t_0^{-1},\ t_1^{-1},\ t_1^{-1},\ t_0^{-2} t_1^{},\ t_0^{} t_1^{-2})$, \\
$\{t_0^{-i} t_1^{-j} t_2^{-k}(t_0^{} t_1^{} \sigma^2, \quad t_0^2 \sigma_1^2, \quad t_1^2 \sigma_0^2) \}
\setminus \{ \chi_0^{} \}$
\end{tabular}
\\
\hline
\end{tabular}
\end{center}
\end{table}

\noindent \\
In table 2A below we have the weight decompositions for fixed points in $\M_{01}$.
In the first column we have the linear form defining $L$.
The quartic $Q = C \cup L$ is given by the equation $X^i Y^j Z^k =0$.

\begin{table}[!hpt]{Table 2A}
\begin{center}
\begin{tabular}{|c|c|}
\hline
$X$
&
\begin{tabular}{l}
$t_0^{-1} (t_1^{},\ t_1^{},\ t_2^{},\ t_2^{},\ t_1^2 t_2^{-1}, \ t_1^{-1} t_2^2)$, \\
$\{ t_0^{-i} t_1^{-j} t_2^{-k} (t_0^{} t_1^{} \sigma^2_0, \quad t_0^{} t_2^{} \sigma^2_0, \quad t_0^2 \sigma^2) \}
\setminus \{ \chi_0^{} \}$
\end{tabular}
\\
\hline
$Y$
&
\begin{tabular}{l}
$t_1^{-1} (t_0^{},\ t_0^{},\ t_2^{},\ t_2^{},\ t_0^2 t_2^{-1}, \ t_0^{-1} t_2^2)$, \\
$\{ t_0^{-i} t_1^{-j} t_2^{-k} (t_0^{} t_1^{} \sigma^2_1, \quad t_1^{} t_2^{} \sigma^2_1, \quad t_1^2 \sigma^2) \}
\setminus \{ \chi_0^{} \}$
\end{tabular}
\\
\hline
$Z$
&
\begin{tabular}{l}
$t_2^{-1} (t_0^{},\ t_0^{},\ t_1^{},\ t_1^{},\ t_0^2 t_1^{-1}, \ t_0^{-1} t_1^2)$, \\
$\{ t_0^{-i} t_1^{-j} t_2^{-k} (t_0^{} t_2^{} \sigma^2_2, \quad t_1^{} t_2^{} \sigma^2_2, \quad t_2^2 \sigma^2) \}
\setminus \{ \chi_0^{} \}$
\end{tabular}
\\
\hline
\end{tabular}
\end{center}
\end{table}

\subsection{Tangent spaces at fixed points in $\M_1$}   %%%%%%%%%%%%% subsection 6.2
\label{6.2}

Given a torus fixed point $[\F] \in \M_1$, the action of $T$ on $\T_{[\F]} \M_1$ can be described as in
the previous subsection.
Let $\rho \colon W_1 \to \M_1$ be the geometric quotient map, $\rho(\varphi) = [{\mathcal Coker}(\varphi)]$.
There exists $\varphi \in \rho^{-1}([\F])$ such that there are morphisms of groups
\[
\alpha \colon T \to \Aut(\OO(-3) \oplus \OO(-1)), \qquad \beta \colon T \to \Aut(2\OO),
\]
\[
\alpha = \left[
\begin{array}{cc}
\alpha_1 & 0 \\
0 & \alpha_2
\end{array}
\right], \qquad \qquad \beta = \left[
\begin{array}{cc}
\beta_1 & 0 \\
0 & \beta_2
\end{array}
\right],
\]
satisfying the property: $t \varphi = \beta(t) \varphi \alpha(t)$ for all $t \in T$.
Indeed, such morphisms are provided at \ref{5.3}.
Denote
\[
{\mathfrak g}_1 = \T_1 G_1 = \big( \End(\OO(-3) \oplus \OO(-1)) \oplus \End(2\OO) \big)/\CC.
\]
We make the identification $\T_{[\F]} \M_1 = \T_{\varphi} W_1/{\mathfrak g}_1$,
where ${\mathfrak g}_1$ is embedded in $\T_{\varphi} W_1$
via the map $(A,B) \mapsto B \varphi - \varphi A$.
Under this identification the action of $T$ on $\T_{[\F]} \M_1$, denoted by $\star$,
is given by formula \ref{6.1.1} and is induced by an action on $\T_{\varphi} W_1$ given by the same formula,
whose weights are represented in the following array (each character appears as many times as is
the dimension of the corresponding eigenspace in $\T_{\varphi} W_1$):
\[
\label{6.2.1}
\tag{6.2.1}
\left[
\begin{array}{rr}
\beta_1^{-1} \alpha_1^{-1} \sigma^3 & \beta_1^{-1} \alpha_2^{-1} \sigma^1 \\
\beta_2^{-1} \alpha_1^{-1} \sigma^3 & \beta_2^{-1} \alpha_2^{-1} \sigma^1
\end{array}
\right].
\]
The induced action on the invariant subspace ${\mathfrak g}_1$ is given by formula \ref{6.1.3}.
Its weights are represented in the following array (same convention as above):
\[
\label{6.2.2}
\tag{6.2.2}
\left[
\begin{array}{lll}
\chi_0^{} & \beta_1^{-1} \beta_2^{} & \chi_0^{} \\
\beta_2^{-1} \beta_1^{} &\chi_0^{} & \alpha_2^{} \alpha_1^{-1} \sigma^2
\end{array}
\right].
\]
Assume that
\begin{multline*}
\varphi = \left[
\begin{array}{cc}
X^i Y^j Z^k & X \\
0 & Y
\end{array}
\right]. \quad \text{Then} \quad t \varphi = \left[
\begin{array}{cc}
t_0^i t_1^j t_2^k X^i Y^j Z^k & t_0^{} X \\
0 & t_1^{} Y
\end{array}
\right] \\
= \left[
\begin{array}{cc}
t_0 & 0 \\
0 & t_1
\end{array}
\right] \varphi \left[
\begin{array}{cc}
t_0^{i-1} t_1^j t_2^k & 0 \\
0 & 1
\end{array}
\right].
\end{multline*}
Tableaux \ref{6.2.1} and \ref{6.2.2} take the form
\[
\left[
\begin{array}{rr}
t_0^{-i} t_1^{-j} t_2^{-k} \sigma^3 & t_0^{-1} \sigma^1 \\
t_0^{1-i} t_1^{-1-j} t_2^{-k} \sigma^3 & t_1^{-1} \sigma^1
\end{array}
\right], \qquad \left[
\begin{array}{lll}
\chi_0^{} & t_0^{-1} t_1^{} & \chi_0^{} \\
t_0^{} t_1^{-1} & \chi_0^{} & t_0^{1-i} t_1^{-j} t_2^{-k} \sigma^2
\end{array}
\right].
\]
We can now list the weights for the action of $T$ on $\T_{[\F]} \M_1$:
\[
\label{6.2.3}
\tag{6.2.3}
t_0^{-i} t_1^{-j} t_2^{-k} \sigma^3_{ijk}, \quad t_0^{1-i} t_1^{-1-j} t_2^{-k} \sigma_1^3, \quad t_0^{-1} t_2^{}, \quad t_1^{-1} t_2^{}.
\]
Assume that
\begin{multline*}
\varphi = \left[
\begin{array}{cc}
0 & X \\
X^i Z^k & Y
\end{array}
\right]. \quad \text{Then} \quad t \varphi = \left[
\begin{array}{cc}
0 & t_0^{} X \\
t_0^i t_2^k X^i Z^k & t_1^{} Y
\end{array}
\right] \\
= \left[
\begin{array}{cc}
t_0^{} & 0 \\
0 & t_1^{}
\end{array}
\right] \varphi \left[
\begin{array}{cc}
t_0^i t_1^{-1} t_2^k & 0 \\
0 & 1
\end{array}
\right].
\end{multline*}
Tableaux \ref{6.2.1} and \ref{6.2.2} take the form
\[
\left[
\begin{array}{rr}
t_0^{-1-i} t_1^{} t_2^{-k} \sigma^3 & t_0^{-1} \sigma^1 \\
t_0^{-i} t_2^{-k} \sigma^3 & t_1^{-1} \sigma^1
\end{array}
\right], \qquad \left[
\begin{array}{lll}
\chi_0^{} & t_0^{-1} t_1^{} & \chi_0^{} \\
t_0^{} t_1^{-1} & \chi_0^{} & t_0^{-i} t_1^{} t_2^{-k} \sigma^2
\end{array}
\right].
\]
The weights for the action of $T$ on $\T_{[\F]} \M_1$ are thus
\[
\label{6.2.4}
\tag{6.2.4}
t_0^{-i} t_2^{-k} \sigma^3_{i0k}, \quad t_0^{-1-i} t_1^{} t_2^{-k} \sigma^3_0, \quad t_0^{-1} t_2^{}, \quad t_1^{-1} t_2^{}.
\]

\begin{proposition}
Let $[\F] \in \M_1$ be a fixed point for the action of $T$.
Then the cotangent space $N$ to $\M_1$ at $[\F]$ can be identified with $\HH^0(\F)^* \otimes \HH^1(\F)$.
The weights for the action of $T$ on $N$ are $\alpha_1 \beta_1 /t_0 t_1 t_2$ and $\alpha_1 \beta_2 / t_0 t_1 t_2$.
\end{proposition}

\begin{proof}
We apply the $\Ext(\_, \F)$-functor to the following exact sequence
in which $\kappa$ denotes the canonical morphism:
\[
0 \to \OO(-3) \oplus \OO(-1) \stackrel{\varphi}{\to} \HH^0(\F) \otimes \OO \stackrel{\kappa}{\to} \F \to 0.
\]
We obtain a surjective map
$
\epsilon \colon \Ext^1(\F,\F) \to \HH^0(\F)^* \otimes \HH^1(\F).
$
Assume that the extension class of a sheaf $\E$ belongs to $\operatorname{Ker}(\epsilon)$.
We have an exact diagram in which the square on the right is a pull-back square:
\[
\xymatrix
{
0 \ar[r] & \F \ar[r] \ar@{=}[d] & \E' \ar[r]^-{\pi'} \ar[d]^-{\kappa'} & \HH^0(\F) \otimes \OO \ar[r] \ar[d]^-{\kappa} & 0 \\
0 \ar[r] & \F \ar[r] &\E \ar[r]^-{\pi} & \F \ar[r] & 0
}.
\]
By hypothesis $\pi'$ has a splitting $\sigma$.
Clearly $\HH^0(\kappa') \circ \HH^0(\sigma) \circ \HH^0(\kappa)^{-1}$ is a splitting of $\HH^0(\pi)$.
From this it easily follows that we can apply the horseshoe lemma to the second row
of the above diagram and to the resolution at the beginning of this proof.
We get a resolution of the form
\[
0 \to (\OO(-3) \oplus \OO(-1)) \oplus (\OO(-3) \oplus \OO(-1)) \stackrel{\psi}{\to} 2\OO \oplus 2\OO \to \E \to 0,
\]
\[
\psi = \left[
\begin{array}{cc}
\varphi & w \\
0 & \varphi
\end{array}
\right].
\]
It is clear now that $\E$ gives a tangent vector to $\M_1$.
This is the vector represented by the image of $w$ in $W_1/{\mathfrak g}_1$.
Thus $\operatorname{Ker}(\epsilon) \subset \T_{[\F]} \M_1$.
Both spaces have dimension $15$, hence they are equal.

To determine the action of $T$ on $N$ consider the commutative diagram
\[
\xymatrix
{
0 \ar[r] & \OO(-3) \oplus \OO(-1) \ar[r]^-{\varphi} \ar[d]^-{\alpha(t)^{-1}}
& 2\OO \ar[r] \ar[d]^-{\beta(t)} & \F \ar[d]^-{\gamma(t)} \ar[r] & 0 \\
0 \ar[r] & \OO(-3) \oplus \OO(-1) \ar[r]^-{t \varphi} & 2\OO \ar[r] & t \F \ar[r] & 0
}.
\]
The action of $T$ on $\HH^0(\F)$ (determined up to a homothety) is given by $t s = \gamma(t)^{-1} \mu_{t^{-1}}^*(s)$.
Under the identification $\CC^2 \simeq \HH^0(\F)$ we have $t s = \beta(t)^{-1}(s)$.
Thus $t$ acts on $\HH^0(\F)^*$ by multiplication with $\beta(t)$.
Let $\F^\D = {\mathcal Ext}^1(\F, \omega_{\PP^2})$ denote the dual sheaf.
By Serre duality we have the isomorphism $\HH^1(\F) \simeq \HH^0(\F^\D)^* \otimes \HH^2(\omega_{\PP^2})$.
By \cite[Lemma 3]{rendiconti}, $\F^\D$ is the cokernel of the transpose of $\varphi$,
hence $t$ acts on $\HH^0(\F^\D)^*$ by multiplication with $\alpha_1(t)$.
The vector space $\HH^2(\omega_{\PP^2})$ is one-dimensional, generated by the monomial $X^{-1} Y^{-1} Z^{-1}$,
so $t$ acts by multiplication with $t_0^{-1} t_1^{-1} t_2^{-1}$.
\end{proof}

\noindent
When
\[
\varphi = \left[
\begin{array}{cc}
X^i Y^j Z^k & X \\
0 & Y
\end{array}
\right], \qquad \text{respectively} \qquad \varphi = \left[
\begin{array}{cc}
0 & X \\
X^i Z^k & Y
\end{array}
\right],
\]
the weights for the action of $T$ on $N$ are
\[
\label{6.2.5}
\tag{6.2.5}
t_0^{i-1} t_1^{j-1} t_2^{k-1}, \quad t_0^{i-2} t_1^j t_2^{k-1}, \quad \text{respectively} \quad
t_0^i t_1^{-2} t_2^{k-1}, \quad t_0^{i-1} t_1^{-1} t_2^{k-1}.
\]

\noindent \\
Combining \ref{6.2.3}, \ref{6.2.4} and \ref{6.2.5} we can determine the weight decomposition
of $\T_{[\F]} \M$ at all fixed points $[\F] \in \M_1$, cf. table 3A below.
In the first column we have the ideal defining $p$. The monomial defining $Q$ is $X^i Y^j Z^k$.

\begin{table}[!hpt]{Table 3A}
\begin{center}
\begin{tabular}{|c|c|}
\hline
$(X, Y)$
&
\begin{tabular}{ll}
$t_0^{-1} t_2^{}, \quad t_1^{-1} t_2^{}, \quad
t_0^i t_1^j t_2^k (t_0^{-2} t_1^{-1} t_2^{-1}, \quad t_0^{-1} t_1^{-2} t_2^{-1})$, \\
$\{ t_0^{-i} t_1^{-j} t_2^{-k}(t_0^{} \sigma^3_1, \quad t_1^{} \sigma^3_0, \quad t_0^{} t_1^{} \sigma^2) \}
\setminus \{ \chi_0^{} \}$
\end{tabular}
\\
\hline
$(X, Z)$
&
\begin{tabular}{ll}
$t_0^{-1} t_1^{}, \quad t_1^{} t_2^{-1}, \quad
t_0^i t_1^j t_2^k (t_0^{-2} t_1^{-1} t_2^{-1}, \quad t_0^{-1} t_1^{-1} t_2^{-2})$, \\
$\{ t_0^{-i} t_1^{-j} t_2^{-k}(t_0^{} \sigma^3_2, \quad t_2^{} \sigma^3_0, \quad t_0^{} t_2^{} \sigma^2) \}
\setminus \{ \chi_0^{} \}$
\end{tabular}
\\
\hline
$(Y, Z)$
&
\begin{tabular}{ll}
$t_0^{} t_1^{-1}, \quad t_0^{} t_2^{-1}, \quad
t_0^i t_1^j t_2^k (t_0^{-1} t_1^{-2} t_2^{-1}, \quad t_0^{-1} t_1^{-1} t_2^{-2})$, \\
$\{ t_0^{-i} t_1^{-j} t_2^{-k}(t_1^{} \sigma^3_2, \quad t_2^{} \sigma^3_1, \quad t_1^{} t_2^{} \sigma^2) \}
\setminus \{ \chi_0^{} \}$
\end{tabular}
\\
\hline
\end{tabular}
\end{center}
\end{table}

Let $\lambda(t)=(t_0^{n_0}, t_1^{n_1}, t_2^{n_2})$ be a one-parameter subgroup of $T$
that is not orthogonal to any non-zero character $\chi$ appearing in tables 1A, 2A, 3A.
Inspecting these tables we see that this condition is equivalent to saying that
$n_0$, $n_1$, $n_2$ are distinct and we do not have any relation of the form
\[
n_i = \frac{n_j}{2} + \frac{n_k}{2}, \qquad n_i = \frac{n_j}{3} + \frac{2 n_k}{3}, \qquad n_i = \frac{n_j}{4} + \frac{3 n_k}{4}
\]
for all distinct indices $i$, $j$, $k$.
For instance, we can choose $\lambda(t)=(1,t,t^5)$.
For each torus fixed point $[\F]$ in $\M$ denote by $p[\F]$ the number of characters $\chi$ in the weight
decomposition of $\T_{[\F]} \M$ satisfying $\langle \lambda, \chi \rangle > 0$.
These numbers can be computed using the {\sc Singular} \cite{singular} programs from Appendix \ref{programs}.
The results are written in tables 1B, 2B, 3B.
When $[\F]$ varies in a projective line $\Lambda_i$ from \ref{5.3}, $p[\F]$ remains unchanged,
so it may be denoted $p(\Lambda_i)$.
Row 4 of table 2B contains the numbers $p(\Lambda_1), \ldots, p(\Lambda_6)$;
the other rows deal only with isolated fixed points.
Let $F \subset \M$ be the set of isolated torus fixed points.
According to \ref{homology_basis}, the Poincar\'e polynomial of $\M$ satisfies the relation
\[
P_{\M}(x) = \sum_{[\F] \in F} x^{2 p[\F]} + \sum_{i=1}^6 (1+x^2) x^{2 p(\Lambda_i)}.
\]
The computations in Appendix \ref{programs} yield the formula for $P_{\M}$ given at \ref{main_theorem}.

\begin{table}[!hpt]{Table 1B}
\begin{center}
\begin{tabular}{|c|c|}
\hline
$(XY, XZ, YZ)$
&
$14, 11, 12, 8, 4, 3, 13, 7, 5, 10, 9, 6$
\\
\hline
$(XY, XZ, Y^2)$
&
$12, 15, 11, 13, 8, 4, 14, 7, 5, 10, 9, 6$
\\
\hline
$(YZ, XZ, Y^2)$
&
$11, 2, 3, 7, 12, 10, 4, 6, 13, 5, 8, 9$
\\
\hline
$(XY, XZ, Z^2)$
&
$2, 10, 13, 4, 3, 11, 7, 12, 5, 9, 6, 8$
\\
\hline
$(XY, YZ, Z^2)$
&
$15, 5, 4, 12, 14, 9, 8, 6, 13, 7, 11, 10$
\\
\hline
$(XY, YZ, X^2)$
&
$15, 5, 4, 12, 14, 9, 8, 6, 13, 7, 11, 10$
\\
\hline
$(XZ, YZ, X^2)$
&
$2, 8, 11, 3, 4, 13, 5, 12, 7, 9, 6, 10$
\\
\hline
$(YZ, Y^2, Z^2)$
&
$10, 1, 11, 7, 3, 2, 12, 6, 4, 9, 8, 5$
\\
\hline
$(XZ, X^2, Z^2)$
&
$14, 3, 13, 11, 4, 5, 12, 6, 8, 9, 10, 7$
\\
\hline
$(XY, X^2, Y^2)$
&
$13, 17, 9, 14, 12, 16, 6, 8, 15, 7, 10, 11$
\\
\hline
\end{tabular}
\end{center}
\end{table}

\begin{table}[!hpt]{Table 2B}
\begin{center}
\begin{tabular}{|c|c|}
\hline
$X$
&
$16, 15, 12, 14, 8, 11, 13, 5$
\\
\hline
$Y$
&
$11, 12, 7, 13, 4, 8, 14, 3$
\\
\hline
$Z$
&
$0, 1, 2, 3, 5, 4, 6, 11$
\\
\hline
&
$9, 9, 7, 6, 5, 9$
\\
\hline
\end{tabular}
\end{center}
\end{table}

\begin{table}[!hpt]{Table 3B}
\begin{center}
\begin{tabular}{|c|c|}
\hline
$(X, Y)$
&
$15, 11, 14, 10, 12, 9, 5, 4, 13, 8, 6, 7$
\\
\hline
$(X, Z)$
&
$14, 3, 10, 13, 5, 4, 11, 8, 7, 12, 6, 9$
\\
\hline
$(Y, Z)$
&
$2, 10, 3, 4, 7, 11, 9, 13, 5, 6, 12, 8$
\\
\hline
\end{tabular}
\end{center}
\end{table}

%%%%%%%%%%%%%%%%%%%%%%%%%%%%%%%%%%%%%%%%%%%%%%%%

\appendix

\section{Singular programs}
\label{programs}

\begin{verbatim}

ring r=0,(x,y,z),dp;
int n;
list s, s0, s1, s2, s2_0, s2_1, s2_2, s3_0, s3_1, s3_2, s4, t0, t1,
t2, t3, t4, t5, t6, t7, t8, t9, t10, t11, t12, t13, t14, t15, t16,
w, T;
poly P;
s0=list(x-y, x-z, -x+y, y-z, -x+z, -y+z);
s1=list(x,y,z);
s2=list(2x, 2y, 2z, x+y, x+z, y+z);
s2_0=list(2y, 2z, y+z);
s2_1=list(2x, 2z, x+z);
s2_2=list(2x, 2y, x+y);
s3_0=list(3y, 2y+z, y+2z, 3z);
s3_1=list(3x, 2x+z, x+2z, 3z);
s3_2=list(3x, 2x+y, x+2y, 3y);

proc addition(poly p, list l)
 {int i; list ll; ll=list();
  for(i=1; i<=size(l); i=i+1){ll=ll+list(p+l[i]);};
  return(ll);};

proc the_values(list l, list ll)
 {int i; list lll; lll=list();
  for(i=1; i<=size(l); i=i+1)
   {lll=lll+list((l[i]/x)*ll[1]+(l[i]/y)*ll[2]+(l[i]/z)*ll[3]);};
  return(lll);};

proc positive_part(list l)
 {int i,p; p=0;
  for(i=1; i<=size(l); i=i+1){if(l[i]>0){p=p+1;};};
  return(p);};

s=list(3x+y, 3x+z, x+3y, 3y+z, x+3z, y+3z, 2x+2y, 2x+2z, 2y+2z,
    2x+y+z, x+2y+z, x+y+2z);
t1=list();
for(n=1; n<=12; n=n+1)
 {w=s0 + addition(-s[n], addition(x+y, s2_2) + addition(x+z, s2_1) +
 addition(y+z, s2_0) + addition(x+y+z, s1));
 t1=t1+list(positive_part(the_values(w, list(0,1,5))));};

s=list(4y, 3x+y, 3x+z, x+3y, 3y+z, x+3z, 2x+2y, 2x+2z, 2y+2z,
    2x+y+z, x+2y+z, x+y+2z);
t2=list(); t3=list(); t4=list(); t5=list(); t6=list(); t7=list();
for(n=1; n<=12; n=n+1)
 {w=list(x-y, x-z, -x+y, -x+z, -y+z) + list(-2y+2z) +
  addition(-s[n], addition(x+y, s2_2) + addition(x+z, s2_1) +
  addition(2y, s2_0) + addition(x+y+z, s1));
  t2=t2+list(positive_part(the_values(w, list(0,1,5))));
  t3=t3+list(positive_part(the_values(w, list(5,1,0))));
  t4=t4+list(positive_part(the_values(w, list(0,5,1))));
  t5=t5+list(positive_part(the_values(w, list(5,0,1))));
  t6=t6+list(positive_part(the_values(w, list(1,0,5))));
  t7=t7+list(positive_part(the_values(w, list(1,5,0))));};

s=list(4y, 4z, x+3y, 3y+z, x+3z, y+3z, 2x+2y, 2x+2z, 2y+2z,
    2x+y+z, x+2y+z, x+y+2z);
t8=list(); t9=list(); t10=list();
for(n=1; n<=12; n=n+1)
 {w=addition(x,list(-y,-y,-z,-z,-2y+z,y-2z)) +
  addition(-s[n],addition(y+z,s2)+addition(2y,s2_2)+addition(2z,s2_1));
  t8=t8+list(positive_part(the_values(w, list(0,1,5))));
  t9=t9+list(positive_part(the_values(w, list(1,0,5))));
  t10=t10+list(positive_part(the_values(w, list(5,1,0))));};

s=list(4x, 3x+y, 3x+z, 2x+2y, 2x+2z, 2x+y+z, x+3y, x+3z);
t11=list(); t12=list(); t13=list();
for(n=1; n<=8; n=n+1)
 {w=addition(-x, list(y, y, z, z, 2y-z, -y+2z)) + addition(-s[n],
  addition(x+y, s2_0) + addition(x+z, s2_0) + addition(2x, s2));
  t11=t11+list(positive_part(the_values(w, list(0,1,5))));
  t12=t12+list(positive_part(the_values(w, list(1,0,5))));
  t13=t13+list(positive_part(the_values(w, list(5,1,0))));};

t0=list();
w=addition(-x, list(y, y, z, z, 2y-z, -y+2z)) + addition(-x-2y-z,
addition(x+y, s2_0) + addition(x+z, s2_0) + addition(2x, s2));
t0=t0+list(positive_part(the_values(w, list(0,1,5))));
t0=t0+list(positive_part(the_values(w, list(1,0,5))));
t0=t0+list(positive_part(the_values(w, list(5,1,0))));
t0=t0+list(positive_part(the_values(w, list(0,5,1))));
t0=t0+list(positive_part(the_values(w, list(1,5,0))));
t0=t0+list(positive_part(the_values(w, list(5,0,1))));

s=list(4x, 4y, 3x+y, 3x+z, x+3y, 3y+z, x+3z, y+3z,
          2x+2y, 2x+2z, 2y+2z, x+y+2z);
t14=list(); t15=list(); t16=list();
for(n=1; n<=12; n=n+1)
 {w=list(-x+z, -y+z, s[n]-2x-y-z, s[n]-x-2y-z) + addition(-s[n],
 addition(x, s3_1) + addition(y, s3_0) + addition(x+y, s2));
 t14=t14+list(positive_part(the_values(w, list(0,1,5))));
 t15=t15+list(positive_part(the_values(w, list(0,5,1))));
 t16=t16+list(positive_part(the_values(w, list(5,1,0))));};

T=t1+t2+t3+t4+t5+t6+t7+t8+t9+t10+t11+t12+t13+t14+t15+t16;
P=0;
for(n=1; n<=6; n=n+1) {P=P+(1+x^2)*(x^(2*t0[n]));};
for(n=1; n<=size(T); n=n+1) {P=P+x^(2*T[n]);};
P;
x34+2x32+6x30+10x28+14x26+15x24+16x22+16x20+16x18+16x16+16x14+16x12+
15x10+14x8+10x6+6x4+2x2+1
\end{verbatim}

\end{document}